\documentclass[12pt,intlimits]{article}

\usepackage{amssymb}
\usepackage{comment}
\usepackage{mathrsfs}
\usepackage{amsthm}
\usepackage{amsmath}
\usepackage[dvipsnames,usenames]{xcolor}

\usepackage{lineno,hyperref}

\usepackage[normalem]{ulem} 

\usepackage{tikz-cd}

\usepackage{enumitem}  

\DeclareFontFamily{U}{mathb}{\hyphenchar\font45}
\DeclareFontShape{U}{mathb}{m}{n}{
<-6> mathb5 <6-7> mathb6 <7-8> mathb7
<8-9> mathb8 <9-10> mathb9
<10-12> mathb10 <12-> mathb12
}{}
\DeclareSymbolFont{mathb}{U}{mathb}{m}{n}
\DeclareMathSymbol{\llcurly}{\mathrel}{mathb}{"CE}
\DeclareMathSymbol{\ggcurly}{\mathrel}{mathb}{"CF}

\DeclareFontFamily{U}{mathx}{\hyphenchar\font45}
\DeclareFontShape{U}{mathx}{m}{n}{
      <5> <6> <7> <8> <9> <10>
      <10.95> <12> <14.4> <17.28> <20.74> <24.88>
      mathx10
      }{}
\DeclareSymbolFont{mathx}{U}{mathx}{m}{n}
\DeclareFontSubstitution{U}{mathx}{m}{n}
\DeclareMathAccent{\widecheck}{0}{mathx}{"71}
\DeclareMathAccent{\wideparen}{0}{mathx}{"75}

\newcommand\Periwinklesout{\bgroup\markoverwith{\textcolor{Periwinkle}{\rule[0.5ex]{2pt}{0.4pt}}}\ULon}
\newcommand\violetsout{\bgroup\markoverwith{\textcolor{violet}{\rule[0.5ex]{2pt}{0.4pt}}}\ULon}

\makeatletter
\DeclareFontFamily{OMX}{MnSymbolE}{}
\DeclareSymbolFont{MnLargeSymbols}{OMX}{MnSymbolE}{m}{n}
\SetSymbolFont{MnLargeSymbols}{bold}{OMX}{MnSymbolE}{b}{n}
\DeclareFontShape{OMX}{MnSymbolE}{m}{n}{
    <-6>  MnSymbolE5
   <6-7>  MnSymbolE6
   <7-8>  MnSymbolE7
   <8-9>  MnSymbolE8
   <9-10> MnSymbolE9
  <10-12> MnSymbolE10
  <12->   MnSymbolE12
}{}
\DeclareFontShape{OMX}{MnSymbolE}{b}{n}{
    <-6>  MnSymbolE-Bold5
   <6-7>  MnSymbolE-Bold6
   <7-8>  MnSymbolE-Bold7
   <8-9>  MnSymbolE-Bold8
   <9-10> MnSymbolE-Bold9
  <10-12> MnSymbolE-Bold10
  <12->   MnSymbolE-Bold12
}{}

\let\llangle\@undefined
\let\rrangle\@undefined
\DeclareMathDelimiter{\llangle}{\mathopen}%
                     {MnLargeSymbols}{'164}{MnLargeSymbols}{'164}
\DeclareMathDelimiter{\rrangle}{\mathclose}%
                     {MnLargeSymbols}{'171}{MnLargeSymbols}{'171}
\makeatother

\makeatletter
\newcommand*{\textoverline}[1]{$\overline{\hbox{#1}}\m@th$}
\makeatother

\makeatletter
\DeclareTextCommand{\textprime}{\encodingdefault}{%
  \mbox{$\m@th'\kern-\scriptspace$}%
}
\makeatother

\newcommand{\real}{\mathbb{R}}
\newcommand{\mytorus}{\mathcal{T}}
\newcommand{\mySemi}{\mathbb{S}}
\newcommand{\cP}{P}

\newtheorem{theorem}{Theorem}[section]
\newtheorem{lemma}[theorem]{Lemma}
\newtheorem{corollary}[theorem]{Corollary}
\newtheorem{proposition}[theorem]{Proposition}

\newtheorem{definition}[theorem]{Definition}

\newcommand{\abs}[1]{\ensuremath{\lvert#1\rvert}}
\newcommand{\norm}[1]{\ensuremath{\lVert#1\rVert}}

\newcommand{\RR}{\ensuremath{\mathbb{R}}}
\newcommand{\NN}{\ensuremath{\mathbb{N}}}

\newcommand{\ZZ}{\ensuremath{\mathbb{Z}}}

\newcommand{\Id}{\ensuremath{\mathrm{Id}}}

\newcommand{\e}{\ensuremath{\mathbf{e}}}

\newcommand{\taumin}{\ensuremath{\tau_{\mathrm{min}}}}
\newcommand{\taumax}{\ensuremath{\tau_{\mathrm{max}}}}

\newcommand{\dd}{\ensuremath{\mathrm{d}}}
\newcommand{\DD}{\ensuremath{\mathrm{D}}}

\DeclareMathOperator{\dist}{dist}
\DeclareMathOperator{\tildedist}{\widetilde{\mathrm{dist}}}

\DeclareMathOperator{\diag}{diag}

\DeclareMathOperator{\card}{card}

\numberwithin{equation}{section}

\newcommand{\ev}{\ensuremath{\mathbf{e}}}

\begin{document}

\author{Stephen Baigent \\
Department of Mathematics \\
University College London
\\
Gower Street, London WC1E 6BT
\\
United Kingdom of Great Britain and Northern Ireland
\\[2ex]
and
\\[2ex]
Janusz Mierczy\'nski \\
Faculty of Pure and Applied Mathematics \\
Wroc{\l}aw University of Science and Technology \\
Wybrze\.{z}e Wyspia\'nskiego 27
\\
PL 50-370 Wroc{\l}aw
\\
Poland
}

\title{Time-periodic carrying simplex for a competitive system of Carathéodory ODEs\thanks{This version of the article has been accepted for publication, after peer review but is not the Version of Record and does not reflect post-acceptance improvements, or any corrections. The version of record of this article, first published in the Journal of Dynamics and Differential Equations, is available online at Publisher’s website: https://doi.org/10.1007/s10884-026-10549-4}}
\date{}

\maketitle

\begin{abstract}
    We establish the existence of a time-periodic carrying simplex for competitive systems of Kolmogorov-type ordinary differential equations of the form $\dot{x}_i = x_i g_i(t,x)$, $i=1,\dots,d$, on $\mathbb{R}^d_+$, where the vector field $g$ is $T$-periodic in time and satisfies only Carath\'eodory conditions rather than the classical $C^1$ regularity assumed in earlier work. This relaxed regularity requires a careful reworking of global existence, uniqueness, and continuous dependence results for the associated initial value problem on the closed first orthant. We show that the time-$T$ Poincar\'e map $P$ generated by the flow is retrotone, and apply recent results on the carrying simplex for retrotone maps to obtain for $P$ a carrying simplex $\Sigma$ that is an unordered, Lipschitz, invariant manifold homeomorphic to the standard probability simplex, and that attracts all nonzero trajectories of the discrete dynamical system generated by $P$. We then extend this construction to the full nonautonomous flow by introducing an auxiliary autonomous system on the extended phase space $\mathcal{T} \times \mathbb{R}^d_+$, where $\mathcal{T}$ is the circle of length $T$, generated via a skew-product formulation. In the context of this extended phase space we then construct a compact global attractor and identify, as its relative boundary, a time-periodic family of carrying simplices $\{\Sigma_s\}_{s \in \mathcal{T}}$ that depend continuously on the phase $s$, each homeomorphic to the standard simplex and each attracting, uniformly in $s$, all nonzero orbits of the corresponding time-shifted flow. We further show that the restriction of the extended flow to the time-periodic carrying simplex is topologically conjugate, via orthogonal projection, to a flow generated by a reduced time-periodic Carath\'eodory system on a $(d-1)$-dimensional linear subspace, extending prior results in the autonomous setting to the nonautonomous, low-regularity case. These results provide a rigorous foundation for the qualitative analysis of seasonally forced competitive population models in which growth rates need only be measurable and locally Lipschitz in time, broadening the applicability of carrying simplex theory to more realistic ecological settings.
\end{abstract}

\bigskip

\section*{Introduction}

Competitive differential equations arise naturally in ecology, epidemiology, and population biology whenever  species  share common resources in such a way that an increase in any one species’ density causes a decrease in the per-capita growth rate of every other competing species. An important feature of many of these competitive ordinary differential equations is the carrying simplex, an invariant manifold of codimension-1 that attracts all nonzero orbits. The carrying simplex was first studied by Morris Hirsch \cite{HirschCS} for $C^1$ autonomous competitive ODEs on $\real_+^d:=[0,\infty)^d$:
\[
\dot{x}_i = x_i f_i(x), \mbox{ where } \frac{\partial f_i(x)}{\partial x_j} \leq 0, \; i,j=1,\ldots,d, \, x\in \real_+^d.
\]
Under conditions that made good ecological sense, Hirsch showed the existence of a Lipschitz invariant manifold of dimensions $d-1$ that attracted all nonzero points and that was unordered, which essentially means decreasing in all directions. All omega limits belonged to the manifold and so the original problem could be reduced to the study of a system of one dimension less. M. L. Zeeman called this manifold the carrying simplex as it could be viewed as an extension of the carrying capacity defined for a single species as the maximum stable population density of the species that the assumed constant environment could support.

Subsequently the notion of the carrying simplex \cite{SmithSIAM,SmithJDE,J-M-W,dMS,WJ2002} was extended to Poincaré maps which arose naturally in the study of $T$\nobreakdash-\hspace{0pt}periodic $C^1$ competitive ODEs
\[
\dot{x}_i = x_i f_i(x), \mbox{ where } \frac{\partial f_i(x)}{\partial x_j} \leq 0, f_i(t,x) = f_i(t+T,x),  \; i,j=1,\ldots,d, \, x\in \real_+^d.
\]
Later maps with similar properties to these Poincar\'{e} maps were studied in their own right and there is a  growing literature on the carrying simplex of maps historically called autonomous `competitive maps', but that more recently have been referred to as retrotone maps (see, for example, \cite{takac1990, Wi, WJ2002, DWY, HirschiBD, Kul2006, Kul2010, RH, Kul2018, NRH2018, Gyll2018, Gyll2019, B-M-JDEA, Hou21, M_convex, Jamieson-Merino-2023}).

There is also a growing literature on the existence of a suitably-defined carrying simplex for nonautonomous maps. The pioneering paper in this direction was by Shen and Wang \cite{Shen-Wang}, where they used skew-product flows to study convergence to a carrying simplex for nonautonomous (not necessarily periodic) and even random $C^1$ competitive systems, a class of competitive systems that is well-suited to many biological systems where there are either seasonal changes, such as in birth and death rates, or random fluctuations, such as in the properties of the environment.

Construction of periodic family of carrying simplices has been applied in the investigation of ecological models with seasonal succession.  The first paper, for two species, appears to be by Hsu and Zhao \cite{Hsu-Zhao-12}, in which  a complete classification of the dynamics of a planar Lotka--Volterra competition model with seasonal succession was obtained. Here the authors considered the period map, that is the flow map evaluated at the period and use monotonicity properties of this map to show that  all orbits of  the period map converge to a fixed point. Based upon the dynamics of the period map and the properties of monotone systems they found conditions of extinction and coexistence and competitive exclusion. However they did not consider the existence of a carrying simplex.

The earliest paper to consider the carrying simplex for nonautonomous competitive systems was by Shen and Wang \cite{Shen-Wang} who also treated random competitive systems in the same context. Using skew product flows defined on $\real_+^2 \times H(g)$, where $H(g)$ is the hull of the time-dependent vector field $g$ they established existence of a carrying simplex defined as a globally attracting (invariant) subset of $\real_+^2 \times H(g)$ under their skew-product flow, and proved that their carrying simplex satisfied the key properties of carrying simplices for autonomous competitive ODEs.

More recently, in \cite{Niu-Wang-Xie-DCDSB-2021} Niu, Wang and Xie extended the planar competition with seasonal succession model of Hsu and Zhao to an arbitrary number of species, noting that the model equations could be viewed as periodic Lotka-Volterra equations (i.e. the model parameters were periodic), and established the existence of a carrying simplex as a codimension-1 invariant manifold that attracts all nonzero orbits. But it appears that the carrying simplex $\Sigma$ referred to in \cite{Niu-Wang-Xie-DCDSB-2021} is that of the Poincar\'{e} map, and $\Sigma$ is not invariant under the full continuous-time flow. Using their carrying simplex they provided a simpler analysis of the planar problem addressed in Hsu and Zhao's pioneering paper \cite{Hsu-Zhao-12}. Subsequently Niu, Wang and Xie \cite{Niu-Wang-Xie-SIAM-2025} and \cite{NWX2} studied the three-dimensional analogue of the problem studied in \cite{Hsu-Zhao-12}.

In the present paper, we consider time-periodic competitive differential equations of Kolmogorov type, but we relax the regularity of the time-dependent per-capita growth rates assumed in \cite{Shen-Wang,Hsu-Zhao-12,Niu-Wang-Xie-DCDSB-2021,NWX2}, and impose much weaker regularity, namely Carathéodory conditions (Section \ref{Sec2}). The relaxation of regularity introduces challenges in establishing global existence of solutions, an issue not encountered in the mentioned four papers. We establish retrotonicity conditions for the periodic flow and define the carrying simplex via the compact attractor of compact sets of an extended flow on $\mathcal{T} \times [0,\infty)^d$ where $\mathcal{T}$ is the interval $[0,T]$ with endpoints identified, and $T$ is the period, and show that it satisfies all the properties of the carrying simplex for autonomous systems discussed in our earlier paper \cite{B-M-JDEA}.

\section{Notation and definitions}\label{sec:sec-notation}
Wherever we use the term `map' we mean a continuous function.

$\lVert \cdot \rVert$ stands for the Euclidean norm, and $\lVert \cdot \rVert_1$ stands for the $\ell_1$\nobreakdash-\hspace{0pt}norm, in $\mathbb{R}^d$.

For $x \in \RR^d$ and $Y \subset \RR^d$ we write
\begin{equation*}
    \dist(x, Y) := \inf\,\{\, \norm{x - y} : y \in Y \,\}.
\end{equation*}
\begin{definition}[Hausdorff metric]
For nonempty compact $X, Y \subset \RR^d$ we denote by $d_H(X, Y)$ their \emph{Hausdorff distance},
\begin{equation*}
    d_H(X, Y) := \max\{ \sup\,\{\, \dist(x, Y) : x \in X \,\}, \sup\,\{\, \dist(y, X) : y \in Y \,\}\}.
\end{equation*}
\end{definition}
For $X \subset \RR^d$ we denote by $\mathcal{P}(X)$ the family of all nonempty compact subsets of $X$.

Below, we state some facts from topology which will be useful in the sequel.
\begin{proposition} \label{prop:hyperspace-compact}
    Assume that $X \subset \RR^d$ is compact.  Then $(\mathcal{P}(X), d_H)$ is a compact metric space.
\end{proposition}
For a proof, see, e.g., \cite[Thm.~4.4.15  on p.~74]{AT}, or \cite[Thm.~4.25 on p.~26]{Kechris}.

\begin{definition}[\textup{(}Painlevé--\textup{)}Kuratowski convergence]
Let $X \subset \RR^d$ be compact.  We say that $A \in \mathcal{P}(X)$ is the \emph{\textup{(}Painlevé--\textup{)}Kuratowski} limit of a sequence $(A_m)_{m = 0}^{\infty} \subset \mathcal{P}(X)$ if the following conditions are satisfied.
\begin{enumerate}
    \item[\textup{(PK1)}]
    For each $a \in A$ there is a sequence $(a^{(m)})_{m = 0}^{\infty}$, $a^{(m)} \in A_m$ for each $m \in \NN$, such that $\lim\limits_{m \to \infty} a^{(m)} = a$.
    \item[\textup{(PK2)}]
    For any sequences $(m_k)_{k = 0}^{\infty}$ and $(a^{(m_k)})_{k = 0}^{\infty}$, $a^{(m_k)} \in A_{m_k}$ for each $k \in \NN$, such that if $\lim\limits_{k \to \infty} m_k = \infty$ and $\lim\limits_{k \to \infty} a^{(m_k)} = a$ then there holds $a \in A$.
\end{enumerate}
\end{definition}
\begin{proposition} \label{prop:Kuratowski-Hausdorff}
    Assume $X \subset \RR^d$ to be compact.  Then $(A_m)_{m = 0}^{\infty} \subset \mathcal{P}(X)$ converges to $A \in \mathcal{P}(X)$ in the Hausdorff metric if and only~if $A$ is the Kuratowski limit of $(A_m)_{m = 0}^{\infty}$.
\end{proposition}
For a proof, see~\cite[Prop.~4.4.14 on p.~73]{AT}.

As in the sequel all sequences $(A_m)$ of sets considered satisfy the assumptions of Proposition~\ref{prop:Kuratowski-Hausdorff}, we will write $\lim\limits_{m \to \infty} A_m = A$.

\begin{proposition} \label{prop:hyperspace-continuous}
    Assume that $X, Y \subset \RR^d$ are compact, and let $h \colon X \to Y$ be continuous.  Then the map $h^{\star} \colon \mathcal{P}(X) \to \mathcal{P}(Y)$ defined as
    \begin{equation*}
        h^{\star}(A) := \{\, h(a) : a \in A  \,\}, \qquad A \in \mathcal{P}(X),
    \end{equation*}
    is continuous.
\end{proposition}
See~\cite[I.4F on pp.~24--28]{Kechris}.

\bigskip
$C_{+} := \{x\in \real^d: x_i\geq 0, \ i=1,\ldots,d\}$, and $C_{++} := \{x\in \real^d: x_i > 0, \ i=1,\ldots,d\}$ will denote convex cones. $C_+$ is often referred to as the \emph{first orthant} and $C_{++}$ is its interior.  $\partial C_{+} = C_{+} \setminus C_{++}$ is called the \emph{boundary} of $C_{+}$ (indeed, it is the boundary of $C_{+}$ in $\RR^d$).

$\ev_i$ is the vector with a one at the $i$\nobreakdash-\hspace{0pt}th position and zeros elsewhere and $\ev =\sum_{i=1}^d\ev_i$, $\hat{\ev}=\ev/\sqrt{d}$.  $\NN =  \{0,1, 2, 3, \dots\}$.

For a subset $I \subset \{1, \ldots,d\}$, let $(C_I)_{+} := \{\, x \in C_{+}: x_i = 0$~for~all~$i \in \{1, \ldots,d\} \setminus I\,\}$ denote a \emph{$k$\nobreakdash-\hspace{0pt}dimensional face} of $C_{I}$, where $k = \card{I}$, and let $(C_I)_{++} := \{\, x \in (C_I)_{+}: x_i > 0$~for~all~$i \in I\,\}$ denote the \emph{relative interior} of $(C_I)_{+}$. A $1$\nobreakdash-\hspace{0pt}dimensional face is referred to as an \emph{$i$\nobreakdash-\hspace{0pt}th axis}, where $I = \{i\}$.

$\bar{A}$ denotes the closure of $A$.

$\Delta := \{x \in C_{+}:\sum\limits_{i=1}^d x_i=1\}$ denotes the standard probability $(d - 1)$\nobreakdash-\hspace{0pt}simplex.  For $I \subset \{1, \ldots,d\}$, $\Delta_{I} := \Delta \cap (C_I)_{+}$.

\begin{definition}[Order relations]
For $x, y \in C_{+}$, we write $x \le  y$ if $x_i \le  y_i$ for all $i=1,\ldots,d$, and $x \ll  y$ if $x_i < y_i$ for all $i=1,\ldots,d$. If $x \le  y$ but $x \ne  y$ we write $x <  y$. The reverse relations are denoted by $\geq, >,\gg$. Two points $x,y\in C_+$ are said to be \emph{order-related} if either $x\leq y$ or $y\leq x$.
\end{definition}

\begin{definition}[Order interval]
For $x, y \in C_{+}$ such that $x \le y$ we define the \emph{order interval} as
\begin{equation*}
   [x, y] := \{\, z \in C_{+} : x \le z \le y \,\}.
\end{equation*}
\end{definition}

\begin{definition}[Orthogonal projection]

For $I \subset \{1, \dots, d\}$, by $\pi_I$ we understand the orthogonal projection of $C_{+}$ onto $(C_I)_{+}$. $\Pi$ denotes the orthogonal projection along $\ev$ onto $V := \ev^{\perp}$.  $\Pi u = u - (u \cdot \hat{\ev}) \hat{\ev} = u - (u \cdot \ev) \ev/d$.
\end{definition}

\begin{definition}[Invariance of sets under maps]
$A \subset C_{+}$ is \emph{$F$\nobreakdash-\hspace{0pt}forward invariant} if $F(A) \subset A$, and \emph{$F$\nobreakdash-\hspace{0pt}invariant} if $F(A) = A$.
\end{definition}
\begin{definition}[Orbits]\mbox{}
\begin{description}
\item[Forward orbit] For $x \in C_{+}$ we denote its \emph{forward orbit}, $O^{+}(x)$, as
\begin{equation*}
    O^{+}(x) := \{\, F^n(x) : n \in \NN \, \}.
\end{equation*}
\item[Backward orbit]
By a \emph{backward orbit} of $x \in C_{+}$ we understand a set $\{\dots, x_{-n-1}, x_{-n}, \dots, \break x_{-2}, x_{-1}, x_{0}\}$ such that $x_0 = x$ and $x_{-n} = F(x_{-n-1})$ for all $n \in \NN$.
\item[Total orbit]
A \emph{total orbit} of $x \in C_{+}$ is the union of a backward orbit of $x$ and the forward orbit, $O^{+}(x)$.
\end{description}
\end{definition}

\begin{definition}
Let $\mathcal{A}$ be a class of subsets of $C_{+}$.  We say that a nonempty compact $F$\nobreakdash-\hspace{0pt}invariant $D \subset C_{+}$ is a \emph{compact attractor, for the dynamical system $(F^n)_{n = 0}^{\infty}$, of the class $\mathcal{A}$} if for any $A \in \mathcal{A}$ and any $\epsilon > 0$ there exists $n_0 = n_0(A, \epsilon)$ such that $\dist(F^n(x), D) < \epsilon$ for all $x \in A$ and all $n \ge n_0$. \textup{(}See \textup{\cite{S-T}} for more details\textup{)}.

\end{definition}

We say that $B \subset C_{+}$ \emph{attracts} the set $A \subset C_{+}$ if for each $\epsilon > 0$ there is $n_0 \in \NN$ such that $\dist(F^n(x), B) < \epsilon$ for all $n \ge n_0$ and all $x \in A$.

\begin{definition}
A compact set $B \subset C_{+}$ is called a \emph{uniform repeller} in $C_{+}$ if there is $\epsilon > 0$ such that $\liminf\limits_{n \to \infty} \dist(F^n(x), B) \ge \epsilon$ for any $x \in C_{+}$.  It is equivalent to the existence of a neighbourhood $U$ of $B$ in $C_{+}$ such that for each $x \in U \setminus B$ there is $n_0 \in \NN$ with the property that $F^n(x) \notin U$ for all $n \ge n_0$, see~\textup{\cite[\emph{Remark~5.15 on p.~136}]{S-T}}.
\end{definition}

\begin{definition}[Compact attractor of bounded sets]
Let $B \subset C_{+}$ be forward invariant.  By the \emph{compact attractor of bounded sets in $B$} we mean a nonempty compact invariant set $\Gamma \subset B$ that attracts any bounded $A \subset B$.  Such a set is unique \textup{(}see~\textup{\cite[\emph{Theorem 2.19 on p.~37}]{S-T}}\textup{)}.  By the \emph{compact attractor of neighbourhoods of compact sets in $B$} we mean a nonempty compact invariant set $\Gamma \subset B$ such that for any compact $A \subset B$ there is a \textup{(}relative\textup{)} neighbourhood $U$ of $A$ in $B$ such that $\Gamma$ attracts $U$.  Such a set is unique \textup{(}see~\textup{\cite[\emph{Theorem 2.19 on p.~37}]{S-T}}\textup{)}.

For $B = C_{+}$ we say simply \emph{compact attractor of bounded sets}, which is the same as the compact attractor of neighbourhoods of compact sets.
\end{definition}

\begin{proposition}
\label{prop:attractor-characterization}
    Let $B \subset C_{+}$ be forward invariant.  The compact attractor of bounded sets is characterized as the set of all $x \in B$ having bounded total orbits.
\end{proposition}
\begin{proof}
  See~\cite[Theorem~2.20 on p.~37]{S-T}.
\end{proof}

\section{Competitive systems of ODEs} \label{Sec2}

For a vector function $g$ of time $t$ and spatial variables $x = (x_1,\ldots,x_d)$ we will by $\mathrm{D}g(t,x)$ understand its derivative matrix with~respect to the latter.

We consider time-periodic systems of ordinary differential equations of the form
\begin{equation}
\label{eq:ODE}
    \dot{x}_i = G_i(t,x): =x_i g_i(t, x), \quad i \in \{1, \ldots, d \}, \ x = (x_1, \ldots, x_d) \in C_{+},
\end{equation}
where $g = (g_1, \ldots, g_d) \colon  \RR \times  C_{+} \to \RR^d$ is periodic in $t$, with period $T > 0$.  Systems of such a type are called periodic \emph{Kolmogorov systems of ODEs}.

We impose Carathéodory conditions on the functions $g$ (for example, see \cite[Sect.~I.5]{Hale-ODEs}) plus periodicity, period $T$, namely:
\medskip
\begin{enumerate}[label=\textup{H\arabic*},ref=\textup{H\arabic*},align=left]
    \item \label{Cara}
    The functions $g_i :\real\times C_{+} \to \real$, $1 \le i \le d$, satisfy
    \begin{enumerate}[label=\textup{H\arabic{enumi}-\alph{enumii}},ref=\textup{H\arabic{enumi}-\alph{enumii}},leftmargin=12pt]
    \item \label{Cara-continuous}
    \quad $g_i(t, \cdot)$ are continuous for almost every $t \in \RR$;
    \item\label{Cara-measurable}
    \quad $g_i(\cdot, x)$ are measurable for each $x \in C_{+}$;
    \item\label{Cara-periodic}
    \quad $g_i(\cdot, x)$ are $T$\nobreakdash-periodic, $T>0$ for each $x\in C_+$;
    \item\label{Cara-bounded}
    \quad for each compact $K \subset C_{+}$ there exists a non-negative $\mu_K \in L^1([0,T])$ such that
    \begin{equation*}
        \abs{g_i(t, x)} \le \mu_K(t)
    \end{equation*}
    for $1 \le i \le d$, almost~all $t \in [0,T]$ and all $x \in K$;
    \item\label{Cara-Lipschitz}
    \quad for each compact $K \subset C_{+}$ there exists a non-negative $\nu_K \in L^1([0,T])$ such that
    \begin{equation*}
        \abs{g_i(t, x) - g_i(t, y)} \le \nu_K(t) \, \norm{x - y}
    \end{equation*}
    for $1 \le i \le d$, almost~all $t \in [0,T]$ and all $x, y \in K$.
    \end{enumerate}
\end{enumerate}
\medskip
Observe that the set of the exceptional values of $t$, that is, of those $t$ for which \ref{Cara-continuous} is not satisfied for some $i$ or \ref{Cara-bounded} or \ref{Cara-Lipschitz} is not satisfied for some $i$ and some compact $K$, has measure zero.  We put the value of $g_i(t, x)$ at those exceptional $t$ to be equal to $1$, for all $i$ and all $x$.  Further, by changing, if~necessary, the values at a set of $t$ with zero measure we can obtain that $g_i(\cdot, x)$ are $T$\nobreakdash-\hspace{0pt}periodic for each $x \in C_{+}$.

From now on we will assume the 5 hypotheses of \ref{Cara} to be satisfied.

Although $g$ is not necessarily continuous in time we have
\begin{lemma}
\label{lm:continuity-integrals}
    Let $h \colon [0, T] \times C_{+} \to C_{+}$ be continuous.  Then the mappings
    \begin{equation*}
        \Bigl[\, C_{+} \ni x \mapsto \int\limits_{0}^{T} g_i(t, h(t, x)) \, \dd t \in \RR \,\Bigr], \quad 1 \le i \le d,
    \end{equation*}
    are continuous.
\end{lemma}
\begin{proof}
    Let $(x^{(k)})_{k = 1}^{\infty} \subset C_{+}$ be a sequence convergent to $x^{(0)}$ as $k \to \infty$. Taking, in~\ref{Cara-bounded}, $K := \{\, h(t, x^{(k)}) : t \in [0, T], k = 0, 1, \ldots \,\}$ we conclude from~\ref{Cara-continuous} and~\ref{Cara-measurable} via the dominated convergence theorem that
    \begin{equation*}
        \int\limits_{0}^{T} g_i(t, h(t, x^{(k)})) \, \dd t \to \int\limits_{0}^{T} g_i(t, h(t, x^{(0)})) \, \dd t \quad \text{as } k \to \infty.
    \end{equation*}
\end{proof}

Observe that it follows from \ref{Cara-bounded} and~\ref{Cara-Lipschitz} that for the vector field $G$ in (\ref{eq:ODE})
\begin{itemize}
    \item
    for each compact $K \subset C_{+}$ there holds
    \begin{equation}
    \label{eq:bounded}
        \abs{G_i(t, x)} \le \sup\{\, \norm{\xi} : \xi \in K \,\} \, \mu_K(t)
    \end{equation}
    for $1 \le i \le d$, almost~all $t \in [0,T]$ and all $x \in K$;
    \item
    for each compact $K \subset C_{+}$ there holds
    \begin{multline}
    \label{eq:Lipschitz}
        \abs{G_i(t, x) - G_i(t, y)} \le \abs{x_i - y_i} \, \abs{g_i(t, x)} + \abs{y_i} \, \abs{g_i(t, x) - g_i(t, y)}
        \\
        \le \bigl(\mu_K(t) + \sup\{\, \norm{\xi} : \xi \in K \,\} \, \nu_K(t)\bigr) \, \norm{x - y}
    \end{multline}
    for $1 \le i \le d$, almost~all $t \in [0,T]$ and all $x, y \in K$.
\end{itemize}
Therefore, if $g$ satisfies \ref{Cara} then $G$ also satisfies the Carathéodory conditions.

Recall (e.g. \cite[Sect.~I.5]{Hale-ODEs}) that a solution of the initial value problem
\begin{equation}
\label{IVP}
    \begin{cases}
        \dot{x} = G(t, x)
        \\
        x(t_0) = x_0,
    \end{cases}
\end{equation}
where $t_0 \in \RR$ and $x_0 \in C_{+}$, is defined as an absolutely continuous function $[J \ni t \mapsto \Phi(t; t_0,x_0) = (\Phi_1(t; t_0,x_0), \ldots, \Phi_d(t; t_0,x_0))]$, where $J \subset \RR$ is an interval containing $t_0$ in its interior, such that
\begin{equation}
\label{eq:integral-equation}
    \Phi(t; t_0, x_0) = x_0 + \int_{t_0}^{t} G(\theta, \Phi(\theta; t_0, x_0)) \, \dd \theta, \quad t \in J.
\end{equation}
Standard results (for example, \cite[Thms.~I.5.1--I.5.3]{Hale-ODEs}) for nonautonomous ODEs on open subsets that satisfy the Carathéodory conditions plus a local Lipschitz condition give existence of unique and noncontinuable solutions on maximal time intervals, but we need to proceed with care as we are working with the closed set $\real\times C_+$.

Recall that the main point of the proof of continuability of solutions is showing that, provided $\taumax(t_0, x_0) < \infty$, if $\lim\limits_{t \nearrow \taumax(t_0, x_0)} \Phi(t; t_0, x_0) = y \in C_{+}$ then the solution can be continued to the right of $\taumax(t_0, x_0)$ (see, e.g., \cite[proof of Thm.~I.2.1 on p. 17]{Hale-ODEs}).  We shall apply that approach to our case.

We partition $C_{+}$ into sets of the form $(C_{I})_{++}$, where $I \subset \{1, \ldots, d\}$.  Each of the sets $(C_{I})_{++}$ can be identified with an open subset of the $\card{I}$\nobreakdash-\hspace{0pt}Euclidean space.  Fix $I$.  Application of \cite[Thms.~I.5.1, I.5.2 and~I.5.3]{Hale-ODEs} gives the existence, for each $t_0 \in \RR$ and each $x_0 \in (C_{I})_{++}$, of a noncontinuable solution $\Phi(\cdot; t_0, x_0) = (\Phi_1(\cdot; t_{0}, x_0), \dots, \Phi_d(\cdot; t_0, x_0))$ defined for $t \in (\taumin(t_0, x_0), \taumax(t_0, x_0))$ with $\taumin(t_0, x_0) < t_0 < \taumax(t_0, x_0)$.  However, so far `noncontinuability' refers to $(C_{I})_{++}$.  Now we can use the Kolmogorov property of~\eqref{eq:ODE} to overcome this difficulty.  Observe that we have, for all $i \in I$,
\begin{equation}
\label{IntEq1}
    \Phi_i(t; t_{0}, x_0) = (x_0)_i \, \exp\biggl(\int_{t_0}^{t} g_i(\theta, \Phi(\theta; t_0, x_0)) \, \dd \theta \biggr), \quad t \in [t_0, \taumax(t_0, x_0)).
\end{equation}
We estimate, by~\ref{Cara-bounded},
\begin{multline*}
    \int_{t_0}^{t} g_i(\theta, \Phi(\theta; t_0, x_0)) \, \dd \theta  \ge - \: \biggl\lvert \int_{t_0}^{t} g_i(\theta, \Phi(\theta; t_0, x_0)) \, \dd \theta \biggr\rvert \ge - \int_{t_0}^{t} \abs{g_i(\theta, \Phi(\theta; t_0, x_0))} \, \dd \theta
    \\
    \ge - \int_{t_0}^{t} \mu_K(\theta) \, \dd \theta \ge - \int_{t_0}^{\taumax(t_0, x_0)} \mu_K(\theta) \, \dd \theta
\end{multline*}
for all $i \in I$ and all $t \in [t_0, \taumax(t_0, x_0))$, where $K$ means the (necessarily compact) closure of $\{\, \Phi(t; t_0, x_0) : t \in [t_0, \taumax(t_0, x_0))\}$.  Consequently, $y = \lim\limits_{t \nearrow \taumax(t_0, x_0)} \Phi(t; t_0, x_0) \in (C_{I})_{++}$, and the solution can be extended, on $(C_{I})_{++}$, to the right of $\taumax(t_0, x_0)$.  In a similar way we treat $\taumin(t_0, x_0)$.

Further, it follows from~\ref{Cara-periodic} via uniqueness that the $T$\nobreakdash-\hspace{0pt}translate of $\Phi(\cdot; t_0, x_0)$ is the (noncontinuable) solution of~\eqref{eq:ODE} satisfying the initial condition $x(t_0 + T) = x_0$.

We collect what we have established so far in the following.
\begin{theorem}
\label{thm:exist}
    For each $t_0 \in \RR$ and each $x_0 \in C_{+}$ there exists a unique noncontinuable solution $\Phi(\cdot; t_0, x_0)$ of~\eqref{IVP} defined on $(\taumin(t_0, x_0), \taumax(t_0, x_0))$, having the property that if $x_0 \in (C_I)_{++}$ then $\Phi(\cdot; t_0, x_0) \in (C_I)_{++}$ for all $t \in (\taumin(t_0, x_0), \taumax(t_0, x_0))$.  Moreover,
    \begin{itemize}
        \item
        if $\taumin(t_0, x_0) > - \infty$ then $\norm{\Phi(t; t_0, x_0)} \to \infty$ as $t \searrow \taumin(t_0, x_0)$, and if $\taumax(t_0, x_0) < \infty$ then $\norm{\Phi(t; t_0, x_0)} \to \infty$ as $t \nearrow \taumax(t_0, x_0)$,
        \item
        for any $t_0 \in \RR$ and any $x_0 \in C_{+}$ one has $\taumin(t_0 + T, x_0) = \taumin(t_0, x_0) + T$, $\taumax(t_0 + T, x_0) = \taumax(t_0, x_0) + T$ and $\Phi(t + T; t_0 + T, x_0) = \Phi(t; t_0, x_0)$ for all $t \in (\taumin(t_0, x_0), \taumax(t_0, x_0))$.
    \end{itemize}
\end{theorem}
We proceed now to the issue of the continuous dependence of solutions on initial values.  We will use \cite[Thm.~I.5.3]{Hale-ODEs}.  Again, the direct application of the theorem is impossible, since it is assumed there that the domain of the function is an open subset of $\RR \times \RR^d$.  However, the main property used in the proof of \cite[Thm.~I.3.1]{Hale-ODEs} (on which \cite[Thm.~I.5.3]{Hale-ODEs} is modelled) is the continuous dependence of solutions of some integral equations on parameters, which follows from our assumptions. Thus we have
\begin{theorem}
\label{thm:cont-depend}
    Assume \ref{Cara}.  Then the mapping
        \begin{equation*}
            \bigl[\, D(\Phi) \ni (t, t_0, x_0) \mapsto \Phi(t; t_0, x_0) \in \RR^d  \,\bigr]
        \end{equation*}
        is continuous on its domain of existence
        \begin{equation*}
            D(\Phi) = \bigcup_{\substack{t_0 \in \RR \\ x_0 \in C_{+}}} (\taumin(t_0, x_0), \taumax(t_0, x_0)) \times \{t_0\} \times \{x_0\}.
        \end{equation*}
\end{theorem}
As a consequence, we formulate the following for continuous dependence on data:

\begin{lemma}
\label{lm:slightly-insane}
    Let $x^{(m)} \to x$ in $C_{+}$, $t_0^{(m)} \to t_0$ and $t^{(m)} \to t$ with $t \in (\taumin(t_0, x), \taumax(t_0, x))$.  Then $t^{(m)} \in (\taumin(t_0^{(m)}, x^{(m)}), \taumax(t_0^{(m)}, x^{(m)}))$ for sufficiently large $k$ and $\Phi(t^{(m)}; t_0^{(m)}, x^{(m)}) \to \Phi(t; t_0, x)$.
\end{lemma}

\medskip

Thus far we have established properties of  $\Phi$ that describe the existence and uniqueness of solutions of the time-periodic ODE (\ref{eq:ODE}). In order to be able to apply the theory of dynamical systems in a standard way we make Eq.~\eqref{eq:ODE} autonomous by attaching an additional dimension corresponding to time translation
\begin{equation}
\label{eq:ODE-time_axis}
    \begin{cases}
        \displaystyle \frac{\mathrm{d}s}{\mathrm{d}t} = 1
        \\[2ex]
        \displaystyle \frac{\mathrm{d}x_i}{\mathrm{d}t} = x_i g_i(s, x).
    \end{cases}
\end{equation}
System~\eqref{eq:ODE-time_axis} generates a local flow (continuous-time dynamical system) $\overline{\Phi}$ on a subset $D(\overline{\Phi})$ of  $\RR \times ( \RR \times C_{+})$, $\overline{\Phi} \colon D(\overline{\Phi}) \to \RR \times C_{+}$, via the formula
\begin{equation}
    \overline{\Phi}(t, (s, x)) := (s + t, \Phi(s + t; s, x)). \label{PhiBar}
\end{equation}
As such, $\overline{\Phi}$ satisfies the following properties of dynamical systems: 
\begin{itemize}
    \item[(EF0)]
    $\overline{\Phi}$ is continuous, and $D(\overline{\Phi})$ is an open subset of $\RR \times (\RR \times C_{+})$ containing $\{\, (0, (s, x)) : s \in \RR, x \in C_{+}\,\}$;
    \item[(EF1)]
    $\overline{\Phi}(0, (s, x)) = (s, x)$ for all $s \in \RR$ and all $x \in C_{+}$;
    \item[(EF2)]
    \begin{equation*}
        \overline{\Phi}(t_2, \overline{\Phi}(t_1, (s,x))) = \overline{\Phi}(t_1 + t_2, (s, x)),
    \end{equation*}
    which is to be interpreted so that if for some $t_1, t_2, s \in \RR$ and $x \in C_{+}$ one of the sides as well as $\overline{\Phi}(t_1, (s, x))$ exist then the other side exists too, and the equality holds.

    (The openness of the domain of $\overline{\Phi}$ is a consequence of Lemma~\ref{lm:slightly-insane}, cf.\ \cite[Rem.~2.4]{Hajek-paper} for one\nobreakdash-\hspace{0pt}sided time. For (EF2), see~\cite[Sect.~II.1]{Hajek-book}.)
\end{itemize}

Consequently, $\Phi$ satisfies the following properties:
\begin{itemize}
    \item[(P0)]
    $\Phi$ is continuous, and its domain $D(\Phi)$ is an open subset of $\RR \times \RR \times C_{+}$ containing $\{\, (s, s, x) : s \in \RR, x \in C_{+}\,\}$;
    \item[(P1)]
    $\Phi(s; s, x) = x$ for any $s \in \RR$ and any $x \in C_{+}$;
    \item[(P2)]
    \begin{equation}
    \label{eq:process}
        \Phi(s + t_1 + t_2; s + t_1, \Phi(s + t_1; s, x)) = \Phi(s + t_1 + t_2; s, x),
    \end{equation}
    which is to be interpreted so that if for some $s, t_1, t_2 \in \RR$ and $x \in C_{+}$ one of the sides as well as $\Phi(s + t_1; s, x)$ exist then the other side exists, too, and the equality holds.
\end{itemize}
$\Phi$ will be referred to as the \emph{local process} generated by~\eqref{eq:ODE}, cf.\ \cite[Chpt.~2]{Kl-Ras}.

\medskip
Since \eqref{eq:ODE} is time-periodic, the local process $\Phi$ has the following additional property:
\begin{itemize}
    \item
    \begin{equation}
    \label{eq:process-periodic}
        \Phi(t + s; s, x) = \Phi(t + s + T; s + T, x),
    \end{equation}
    which is to be interpreted so that if for some $t, s \in \RR$ and some $x \in C_{+}$ one side exists then the other side exists, too, and the equality holds.
\end{itemize}
Observe that the solution of~\eqref{eq:ODE} passing through $x$ is $T$\nobreakdash-\hspace{0pt}periodic if and only if $\Phi(T; 0, x) = x$.

\medskip
By (P2), we have the formula
\begin{equation}
\label{eq:process-inverse}
    \Phi(s; t, \Phi(t; s, x)) = x,
\end{equation}
provided that $\Phi(t; s, x)$ exists.  In~particular, it follows that for any $s, t \in \RR$ the map $\Phi(t; s, \cdot)$ is injective on its domain.

\bigskip
We will need to write the left-hand side of~\eqref{eq:process-inverse} in terms of a solution of system~\eqref{eq:ODE} with $g$ replaced with $- g$.  For $x \in C_{+}$ and $t \in \RR$ fixed such that $\Phi(t; 0, x)$ exists we put $\phi(s) := \Phi(s; 0, x)$, $s \in (\taumin(0, x), \taumax(0, x))$.  Let $\chi(\theta) := \phi(t - \theta)$.  We have, where $\dot{} = \frac{\mathrm{d}}{\mathrm{d}{\theta}}$,
\begin{multline*}
    \frac{\mathrm{d}\chi}{\mathrm{d}\theta}(\theta) = - \dot{\phi}(t - \theta) = - G(t - \theta, \phi(t -\theta)) = - G(t - \theta, \chi(\theta)),
    \\
    \theta \in (t - \taumax(0, x), t - \taumin(0, x)).
\end{multline*}
For further reference, we formulate the result of the above calculation as the following.
\begin{lemma}
\label{lm:inverse}
    Let $x \in C_{+}$ and $t \in \RR$ be such that $\Phi(t; 0, x)$ exists.  Then $x$ is equal to the value at time $\theta = t$ of the solution of the initial value problem
\begin{equation}
\label{eq:ODE-retro}
     \begin{cases}
         \displaystyle \frac{\mathrm{d}\xi_i}{\mathrm{d}\theta} = - \xi_i g_i(t - \theta, \xi), \quad i \in \{1, \ldots, d \}, \ \xi \in C_{+},
         \\[1.5ex]
         \xi(0) = \Phi(t; 0, x).
     \end{cases}
\end{equation}
\end{lemma}

\section{Properties of the process $\Phi$}
Recall that \ref{Cara} is our standing assumption.

In this section we put further assumptions on \eqref{eq:ODE}.

\begin{enumerate}[label=\textup{H\arabic{enumi}},ref=\textup{H\arabic{enumi}}]
\setcounter{enumi}{1}
\item \label{0-repel}
\begin{equation*}
   \int\limits_{0}^{T} g_i(t, 0) \, \dd t > 0
\end{equation*}
for all $1 \le i \le d$.  In an ecological context, this says that the average per~capita growth rate of each species remains positive when the population is arbitrarily small.

    \item \label{new-H3}
    For each $i \in \{1, \dots, d\}$, the solution $\Phi(\cdot; 0, \e_i)$ is $T$\nobreakdash-\hspace{0pt}periodic.
\end{enumerate}
From~\ref{new-H3} it follows, via~\eqref{IntEq1}, that, for each $i \in \{1, \dots, d\}$ there holds
\begin{equation}
\label{eq:integral-is-zero}
    \int_{t}^{t + T} g_i(\theta; \Phi(\theta; t, \e_i)) \, \dd \theta = 0
\end{equation}
for any $t \in \RR$.

An easy to check sufficient condition for satisfying \ref{new-H3} is given by the following.

\begin{enumerate}[label=\textup{$\widetilde{\textup{H}}$\arabic{enumi}},ref=\textup{$\widetilde{\textup{H}}$\arabic{enumi}}]
\setcounter{enumi}{2}
\item \label{tilde-H3}
There exists $M > 0$ such that $g_i(t, M \e_i) \le 0$ for all $1 \le i \le d$ and a.a.\ $t \in [0, T]$.
\end{enumerate}

\begin{enumerate}[resume*,label=\textup{H\arabic{enumi}},ref=\textup{H\arabic{enumi}}]
\setcounter{enumi}{3}
\item \label{competitive}
For almost all $t \in [0, T]$ the functions $g_i(t, \cdot)$ are non-increasing in $x_j$, for any $i, j \in \{1, \ldots, d\}$.  Moreover, for each $i \in \{1, \ldots, d\}$ there is a set $\Theta_{i} \subset [0, T]$ of positive measure such that for each $t \in \Theta_{i}$ the functions $g_i(t, \cdot)$ are (strictly) decreasing in $x_i$.   Ecologically this hypothesis says that all species are competing with each other including themselves.
\end{enumerate}
Sometimes we make an assumption stronger than \ref{competitive}.
\begin{enumerate}[label=\textup{H\arabic{enumi}$'$},ref=\textup{H\arabic{enumi}$'$}]
\setcounter{enumi}{3}
    \item \label{strongly-competitive}
    For almost all $t \in [0, T]$ the functions $g_i(t, \cdot)$ are non-increasing in $x_j$, for any $i, j \in \{1, \ldots, d\}$.  Moreover, for any $i, j \in \{1, \ldots, d\}$ there is a set $\Theta_{i,j} \subset [0, T]$ of positive measure such that for each $t \in \Theta_{i,j}$ the functions $g_i(t, \cdot)$ are (strictly) decreasing in $x_j$, for any $i, j \in \{1, \ldots, d\}$.
\end{enumerate}

\subsection{The behaviour of (\ref{eq:ODE}) on the axes}
\label{subsect:Poincare-axes}

For $1 \le i \le d$ define $\tilde{g}_i \colon \RR \times [0, \infty) \to \RR$ by putting $\tilde{g}_i(t, x) := g_i(t, x \ev_i)$.

The restriction of~\eqref{eq:ODE} to the $i$\nobreakdash-\hspace{0pt}th axis $(C_i)_{+}$ corresponds to the scalar ODE
\begin{equation}
\label{eq:ODE-i}
    \dot{x} = x \tilde{g}_i(t,x).
\end{equation}
For $1 \le i \le d$, $t_0 \in \RR$ and $x_0 \in [0, \infty)$ we denote by $\phi_i(\cdot; t_0, x)$ the unique noncontinuable solution of the initial value problem
\begin{equation}
\label{IVP-i}
    \begin{cases}
        \dot{x} = x \tilde{g}_i(t,x)
        \\
        x(t_0) = x_0.
    \end{cases}
\end{equation}
We have
\begin{equation*}
    \phi_i(t; t_0, x) = \Phi_i(t; t_0, x \e_i).
\end{equation*}

Observe that each solution of~\eqref{eq:ODE-i} is either identically zero or takes positive values only.

\begin{lemma} \label{lemma:axes-complete}
    Assume~\ref{Cara}, \ref{new-H3} or~\ref{tilde-H3}, and~\ref{competitive}.  For any $1 \le i \le d$, any $x_0 \in [0, \infty)$ and any $t_0 \in \RR$ the solution $\phi_i(t;t_0,x_0)$ is defined for all $t \ge t_0$.
\end{lemma}
\begin{proof}
    The case $x = 0$ is straightforward, so assume $x > 0$.  Suppose to the contrary that for some $1 \le i \le d$, $t_0 \in \RR$ and $x_0 > 0$ the right end (call it $\taumax$) of the domain of the noncontinuable solution $\phi_i(\cdot; t_0, x_0)$ belongs to $(t_0, \infty)$.  A standard result in the theory of Carathéodory ODEs, see, e.g., \cite[Thm.~I.5.2]{Hale-ODEs}, implies that there exists $t_1 \in [t_0, \taumax)$ such that $\phi_i(t; t_0, x_0) \ge 2$ (under \ref{new-H3}) or $\ge M + 1$ (under \ref{tilde-H3}), for all $t \in [t_1, \taumax)$.  Put $x_1 := \phi_i(t; t_0, x_0)$.

    By~\eqref{eq:ODE-i}, we have
    \begin{equation*}
        \int_{t_1}^{\taumax} \tilde{g}_i(t, \phi_i(t; t_0, x_0)) \, \dd t = \lim\limits_{t \nearrow \taumax} \ln{\phi_i(t; t_0, x_0)} - \ln{x_1},
    \end{equation*}
    which  is $\infty$.  Let $\ell \in \NN \cup \{0\}$ be such that $t_1 + \ell T \le \taumax$ but $t_1 + (\ell + 1) T > \taumax$.   We represent
    \begin{equation*}
        \int_{t_1}^{\taumax} \tilde{g}_i(t, \phi_i(t; t_0, x_0)) \, \dd t = \sum_{k = 0}^{\ell - 1} \int_{t_1 + k T}^{t_1 + (k + 1) T} \tilde{g}_i(t, \phi_i(t; t_0, x_0)) \, \dd t + \int_{t_1  + \ell T}^{\taumax} \tilde{g}_i(t, \phi_i(t; t_0, x_0)) \, \dd t,
    \end{equation*}
    where the first summand on the RHS is put to be $0$ if $\ell = 0$.  We estimate
    \begin{equation*}
        \int_{t_1 + k T}^{t_1 + (k + 1) T} \tilde{g}_i(t, \phi_i(t; t_0, x_0)) \, \dd t \mathrel{\overset{\ref{competitive}}{\le}} \int_{t_1 + k T}^{t_1 + (k + 1) T} \tilde{g}_i(t, x_1) \, \dd t \mathrel{\overset{\ref{Cara-periodic}}{=}} \int_{0}^{T} \tilde{g}_i(t, x_1) \, \dd t \mathrel{\overset{\ref{competitive}}{<}} 0,
    \end{equation*}
    and
    \begin{multline*}
        \int_{t_1  + \ell T}^{\taumax} \tilde{g}_i(t, \phi_i(t; t_0, x_0)) \, \dd t \mathrel{\overset{\ref{competitive}}{\le}} \int_{t_1  + \ell T}^{\taumax} \tilde{g}_i(t, x_1) \, \dd t \mathrel{\overset{\ref{Cara-periodic}}{=}} \int_{0}^{\taumax - (t_1 + \ell T)} \tilde{g}_i(t, x_1) \, \dd t
        \\
        \le \int_{0}^{\taumax - (t_1 + \ell T)} \abs{\tilde{g}_i(t, x_1)} \, \dd t \le \int_{0}^{T} \abs{\tilde{g}_i(t, x_1)} \, \dd t \mathrel{\overset{\ref{Cara-bounded}}{\le}} \int_{0}^{T} \mu_K(t) \, \dd t < \infty,
    \end{multline*}
    where $K$ equals $\{2\e_i\}$ in case of \ref{new-H3} and $\{(M + 1) \e_i\}$ in case of \ref{tilde-H3}.  We have obtained a contradiction.
\end{proof}
Lemma~\ref{lemma:axes-complete} allows us to define, for each $1 \le i \le d$, a mapping $\hat{H}_i \colon [0, \infty) \to [0, \infty)$ by the formula
\begin{equation*}
    \hat{H}_i(x) := \phi_i(T; 0, x), \quad x \in [0, \infty).
\end{equation*}
By~(P0), $\hat{H}_i$ is continuous, and, by Lemma~\ref{lm:inverse}, $\hat{H}_i$ is injective and its inverse is continuous, too.  As a consequence, $\hat{H}_i$ is strongly monotone, and, since $\hat{H}_i(0) = 0$, $\hat{H}_i$ is increasing.

We write $\hat{H}_i^n$ for the $n$\nobreakdash-\hspace{0pt}th iterate of $\hat{H}_i$.

\begin{lemma}
\label{lm:origin-repel}
    Assume~\ref{Cara} and~\ref{0-repel}.  Then there is $\delta > 0$ such that $\hat{H}_i(x) > x$ for all $1 \le i \le d$ and any $x \in (0, \delta)$.
\end{lemma}
\begin{proof}
    It follows from \ref{0-repel} and Lemma~\ref{lm:continuity-integrals}.
\end{proof}

\begin{lemma}
\label{lm:M-goes-down}
    Assume~\ref{Cara}, \ref{tilde-H3} and~\ref{competitive}.  Then $\hat{H}_i(M) \le M$ for all $1 \le i \le d$.
\end{lemma}
\begin{proof}
    Fix $i$.  Suppose to the contrary that $\hat{H}_i(M) > M$.  By the continuity of the solution and the fact that $\phi_i(0; 0, M) = M$, we can find $t_1 \in (0, T)$ such that $\phi_i(t_1; 0, M) = (M + \hat{H}_i(M))/2$ and $\phi_i(t; 0, M) > (M + \hat{H}_i(M))/2$ for all $t \in (t_1, T]$.  We have
    \begin{equation*}
        0 < \ln{\hat{H}_i(M)} - \ln{\phi_i(t_1; 0, M)} \mathrel{\overset{\eqref{eq:ODE-i}}{\le}} \int_{t_1}^{T} \tilde{g}_i(\theta, \phi_i(\theta; 0, M)) \, \dd \theta \mathrel{\overset{\eqref{competitive}}{\le}} \int_{t_1}^{T} \tilde{g}_i(\theta, M) \, \dd \theta \mathrel{\overset{\eqref{tilde-H3}}{\le}} 0,
    \end{equation*}
    a contradiction.
\end{proof}

\begin{proposition}
\label{prop:axes}
    Assume~\ref{Cara}, \ref{new-H3} or~\ref{tilde-H3}, and~\ref{competitive}.  For each $1 \le i \le d$ the following holds.
    \begin{enumerate}
        \item[\textup{(i)}]
        There exists exactly one fixed point, $\hat{x}_i$, of $\hat{H}_i$ in $(0, \infty)$, and it belongs to $(0, M]$.
        \item[\textup{(ii)}]
        For any $x \in (0, \hat{x}_i)$ the sequence $(\hat{H}_i^n(x))_{n = 1}^{\infty}$ strictly increases to $\hat{x}_i$, and for any $x \in (\hat{x}_i, \infty)$ the sequence $(\hat{H}_i^n(x))_{n = 1}^{\infty}$ strictly decreases to $\hat{x}_i$.
    \end{enumerate}
\end{proposition}

\begin{proof}
    We start by showing that for each $1 \le i \le d$ there is at~most one fixed point of $\hat{H}_i$ in $(0, \infty)$.  Suppose to the contrary that there are two such points, $0 < \xi < \eta$.  It follows from the uniqueness of the initial value problem that $\phi_i(t; 0, \xi) < \phi_i(t; 0, \eta)$ for all $t \in \RR$, so
    \begin{align*}
        0 = \ln{\eta} - \ln{\xi} = & \ln{\phi_i(T; 0, \eta)} - \ln{\phi_i(T; 0, \xi)}
        \\
        = & \int_{0}^{T} (\tilde{g}_i(t, \phi_i(t; 0, \eta) - \tilde{g}_i(t, \phi_i(t; 0, \xi) \, \dd t \mathrel{\overset{\eqref{competitive}}{<}} 0,
    \end{align*}
   a contradiction.

   Under~\ref{new-H3} we take $\hat{x}_i = 1$.  Assume now~\ref{tilde-H3}.  As $\hat{H}_i$ is (strictly) increasing, the sequence $(\hat{H}_i^n(M))_{n = 1}^{\infty}$ either is constant (then we put $\hat{x}_i = M$) or (strictly) decreases.  In the latter case its limit (which we denote by $\hat{x}_i$) is, by Lemma~\ref{lm:origin-repel}, $\ge \delta > 0$.

  The reasoning leading to (ii) is standard.
\end{proof}

Let $\psi_i$ stand for the solution of~\eqref{eq:ODE-i} taking value $\hat{x}_i$ at $t = 0$: $\psi_i(t) = \phi_i(t; 0, \hat{x}_i)$.  $\psi_i$ is $T$\nobreakdash-\hspace{0pt}periodic.  As a corollary of Proposition~\ref{prop:axes} and the uniqueness of solutions of initial value problem and their continuous dependence on initial conditions (Theorem~\ref{thm:cont-depend}) we obtain the following.
\begin{proposition}
\label{prop:axes-ODE}
    Let $1 \le i \le d$, and let $\phi_i(\cdot)$ be a generic solution of~\eqref{eq:ODE-i} such that $\phi_i(0) > 0$.  Then there holds:
    \begin{itemize}
        \item
        \begin{equation*}
              \lim\limits_{t \to \infty} \phi_i(t + \tau) = \psi_i(\tau) \quad \text{for any } \tau \in [0, T];
        \end{equation*}
        \item
        $\phi_i(t) < \phi_i(t + T) < \psi_i(t)$ for all $t \ge 0$, provided $0 < \phi_i(0) < \hat{x}_i$;
        \item
        $\psi_i(t) < \phi_i(t + T) < \phi_i(t)$ for all $t \ge 0$, provided $\phi_i(0) > \hat{x}_i$.
        \end{itemize}
\end{proposition}

\bigskip
From now on, when we speak of hypothesis \ref{tilde-H3}, we always assume that system~\eqref{eq:ODE} is, if needed, normalized in the following way: the original $x$ is replaced by $\diag[\hat{x}_i, \ldots, \hat{x}_d]^{-1} x$ and the original $f(x)$ is replaced by $f(\diag[\hat{x}_i, \ldots, \hat{x}_d]^{-1}x)$.  Such normalization takes \ref{tilde-H3} into \ref{new-H3}, leaves \ref{0-repel} and \ref{competitive} (or \ref{strongly-competitive}) unchanged, and leaves functions $\mu_K$ and $\nu_K$ in \ref{Cara} multiplied by a positive constant.  Accordingly, in all the proofs below we will consider the case \ref{new-H3} only.

\subsection{Global existence of the process $\Phi$ in time}
\label{subsect:axes}
\begin{lemma}
\label{lm:new:complete}
    Assume~\ref{Cara}, \ref{new-H3} or \ref{tilde-H3}, and~\ref{competitive}.  For any $x \in C_{+}$ and any $t_0 \in \RR$ the solution $\Phi(t;t_0,x)$ is defined for all $t \ge t_0$.
\end{lemma}
\begin{proof}
    For $x\in C_+$,  we know that $\Phi(t;t_0,x)$ exists for $t\in (\tau_{\min}(t_0,x),\tau_{\max}(t_0,x))$.
    Now, by Theorem~\ref{thm:exist}, if $\Phi(t; t_0, x)$ is not defined for all $t\geq t_0$ there must be some $k$ such that $\Phi_k(t;t_0,x)\to \infty$ as $t\to \tau_{\max} := \tau_{\max} (t_0,x) < \infty$.   Choose such a $k$, and set $L := \sup\{\, \Phi_k(t; 0, \e_k) : t \in [0, T] \,\}$. By~\ref{new-H3}, $0 < L < \infty$.  Let $t_1 \in [t_0, \taumax)$ be such that $\Phi_k(t;t_0,x) > L$ for all $t \in (t_1, \taumax)$.  Put $y := \Phi(t_1;t_0,x)$.  We have $\Phi(t;t_0,x) = \Phi(t;t_1,\Phi(t_1;t_0,x)) = \Phi(t;t_1,y)$.  Recall that $t_1$ has been so chosen that $\Phi(\theta; t_1, y) \ge \Phi(\theta;t_1, \e_k)$ for all $t \in [t_1, \taumax)$, consequently, by the first part of~\ref{competitive}, $g_k(\theta, \Phi(\theta;t_1,y)) \le g_k(\theta, \Phi(\theta;t_1,\e_k))$ for a.a. $\theta \in [t_1, \taumax)$, so that
    \begin{align*}
     \Phi_k(t;t_0,x)=\Phi_k(t;t_1,y) & = y_k \exp \left( \int_{t_1}^t g_k(\theta,\Phi(\theta;t_1,y)) \, \dd \theta\right) \\
     & \le  y_k \exp \left(\int_{t_1}^t g_k(\theta, \Phi(\theta;t_1,\e_k)) \, \dd\theta\right)
    \end{align*}
    for $t \in [t_1,\taumax)$.

    Let $\ell \in \NN \cup \{0\}$ be such that $t_1 + \ell T \le \taumax$ but $t_1 + (\ell + 1) T > \taumax$. By~\eqref{eq:integral-is-zero} and the time\nobreakdash-\hspace{0pt}periodicity of both $g_k$ and the solution,
      \begin{equation*}
          \begin{aligned}
              \int_{t_1}^{\taumax} g_k(\theta, \Phi(\theta;t_1,\e_k)) \, \dd \theta & =  \ell  \int_{0}^{T} g_k(\theta, \Phi(\theta;t_1,\e_k)) \, \dd \theta + \int_{t_1 + \ell T}^{\taumax} g_k(\theta, \Phi(\theta;t_1,\e_k)) \, \dd \theta
              \\
              & = \int_{t_1 + \ell T}^{\taumax} g_k(\theta, \Phi(\theta;t_1,\e_k)) \, \dd \theta = \int_{0}^{\taumax - (t_1 + \ell T)} g_k(\theta, \Phi(\theta;0,\e_k)) \, \dd \theta,
          \end{aligned}
      \end{equation*}
      whose absolute value is, by~\ref{Cara-bounded}, bounded above by
      \begin{equation*}
          \int_{0}^{T} \mu_{K}(\theta) \, \dd \theta < \infty,
      \end{equation*}
      with $K = \{\, \Phi(\theta;0,\e_k) : \theta \in [0, T] \,\}$.  It follows that
      \begin{equation*}
          \exp \left(\int_{t_1}^{\taumax}  g_k(\theta,\Phi(\theta;t_1,y)) \, \dd \theta \right)  < \infty,
      \end{equation*}
      a contradiction.
\end{proof}

\medskip
\begin{lemma}
\label{lm:homeo}
    Assume~\ref{Cara}, \ref{new-H3} or \ref{tilde-H3}, and~\ref{competitive}.  For any $t \ge t_0$, the mapping $\Phi(t; t_0, \cdot)$ is a homeomorphism of $\RR^d$ onto its image.
\end{lemma}
\begin{proof}
    Fix $t \ge t_0$.  By Lemma~\ref{lm:new:complete}, the mapping $\Phi(t; t_0, \cdot)$ is defined on the whole of $C_{+}$, and its continuity is a direct consequence of Theorem~\ref{thm:cont-depend}.  The existence and continuity of the inverse mapping follow from Theorem~\ref{thm:cont-depend} applied to~\eqref{eq:ODE-retro}.
\end{proof}

\subsection{Order preserving of the process $\Phi$ in negative time}
\label{subsect:order-preserving}

\begin{proposition}[Retrotonicity of the periodic flow]
\label{prop:retrotone}
    Assume \ref{Cara}, \ref{new-H3} or \ref{tilde-H3}, and~\ref{competitive}.  Let $x, y \in C_{+}$, $t > 0$, be such that $\Phi(t; 0, x) < \Phi(t; 0, y)$.  Then $x < y$.  Moreover, the following holds.
    \begin{enumerate}
        \item[\textup{(i)}]
        For any $i \in \{1, \ldots, d\}$, if $\Phi_i(t; 0, x) < \Phi_i(t; 0, y)$ then $x_i < y_i$.
        \item[\textup{(ii)}]
        Assume \ref{strongly-competitive}.  For any $i \in \{1, \ldots, d\}$ one has $x_i < y_i$ provided $\Phi_i(T; 0, y) > 0$.
    \end{enumerate}
\end{proposition}
\begin{proof}

Let $y^{(m)} \to y$ as $m \to \infty$ be such that $\Phi(t;0,x) \ll \Phi(t;0,y^{(m)})$ for all $m$. This is possible as $\Phi(t;0,\cdot)$, by Lemma~\ref{lm:homeo}, is a homeomorphism. If $x \not < y$, either $x = y$, which is easily discounted, or there is a nonempty $K \subset \{1, \ldots, d\}$ such that for any $k \in K$ there holds $x_k > (y^{(m)})_k$ for $m$ large enough, say $m > M(k)$.

Since $\Phi_k(t;0,x) < \Phi_k(t;0,y^{(m)}) $ and $\Phi_k(0;0,x) > \Phi_k(0;0,y^{(m)})$ for $m > M(k)$ ($M(k)$ chosen larger if necessary), there exists a $\delta_{m,k} \mathrel{\in} (0, t)$ such that $\Phi_k(t-\delta_{m,k};0,x)= \Phi_k(t-\delta_{m,k};0,y^{(m)})$ and $\Phi_k(\theta;0,x) < \Phi_k(\theta;0,y^{(m)})$ for all $\theta \in (t - \delta_{m,k}, t]$.   Put $l$ to be such that $\delta_{m} := \delta_{m, l} = \min\{\delta_{m, k} : k \in K\}$.  By construction, $\Phi_l(t- \delta_m; 0, x) = \Phi_l(t- \delta_m; 0, y^{(m)})$ and $\Phi_i(\theta;0,x) < \Phi_i(\theta;0,y^{(m)})$ for all $i \in \{1, \ldots, d\}$ and all $\theta \in (t - \delta_{m}, t]$.

By (\ref{IntEq1})
\begin{align*}
&\Phi_l(t; 0, x) = \Phi_l(t-\delta_m;0,x) \exp \left( \int_{0}^{\delta_m} g_l\big(t-\theta,\Phi(t-\theta;0,x)\big)\, \dd \theta
\right)
\\
& \mathrel{=}  \Phi_{l}(t-\delta_{m};0,y^{(m)}) \exp \left( \int_{0}^{\delta_{m}} g_{l}\big(t-\theta,\Phi(t-\theta;0,x)\big)\, \dd\theta \right)\,\,,
\end{align*}
but
\begin{equation*}
    \Phi_l(t; 0, x) < \Phi_l(t; 0, y^{(m)}) \mathrel{\overset{\eqref{IntEq1}}{=}} \Phi_{l}(t-\delta_{m};0,{y^{(m)}}) \exp \left(
\int_{0}^{\delta_{m}} g_{l}\big(t-\theta,\Phi(t-\theta;0,{y^{(m)}})\big)\, \dd \theta
\right)\,\,.
\end{equation*}

Thus (recall that $\Phi_l(t; 0, y^{(m)}) > 0$, so, by the first part of Theorem~\ref{thm:exist}, $\Phi_l(t - \delta_{m}; 0, y^{(m)}) > 0$)
\begin{equation} \label{eq:SAB3}
    \exp \left( \int_{0}^{\delta_{m}} g_{l}\big(t-\theta,\Phi(t-\theta;0,x)\big)\, \dd \theta - \int_{0}^{\delta_m} g_{l}\big(t-\theta,\Phi(t-\theta;0,y^{(m)})\big)\, \dd\theta \right) < 1.
\end{equation}
On the other hand, by \ref{competitive}, $g_{l}\big(t-\theta,\Phi(t-\theta;0,{y^{(m)}})\big) \mathrel{\le} g_{l}\big(t-\theta,\Phi(t-\theta;0,x)\big) $  for a.a. $\theta\in [0,\delta_{m}]$, which contradicts (\ref{eq:SAB3}). Hence $x_{i} \le (y^{(m)})_{i}$ for all $i \in \{1, \ldots, d\}$, consequently $x_{i} \le y_{i}$ for all $i \in \{1, \ldots, d\}$.  As $x \ne y$, we have $x < y$.

 It remains to prove that $x_i < y_i$ for those $i$ for which $\Phi_i(t; 0, x) < \Phi_i(t; 0, y)$.  Suppose $x_i = y_i$ but $\Phi_i(t; 0, x) < \Phi_i(t; 0, y)$.   The case $x_i = y_i = 0$ is almost obvious, so assume $x_i = y_{i} > 0$.  By~\eqref{IntEq1},
 \begin{equation*}
     \int_{0}^{t} \bigl( g_i(\theta, \Phi(\theta; 0, x)) - g_i(\theta, \Phi(\theta; 0, y)) \bigr) \, \dd \theta < 0.
 \end{equation*}

But we have, from the previous reasoning, that $\Phi(\theta; 0, x) \le \Phi(\theta; 0, y)$ for all $\theta \in [0, t]$, so, by~\ref{competitive}, $g_i(\theta, \Phi(\theta; 0, x)) \ge g_i(\theta, \Phi(\theta; 0, y))$.  The contradiction obtained concludes the proof of part~(i).

\medskip

We proceed to the proof of part~(ii).  We have
\begin{equation}
    0 < \Phi_i(T; 0, y) = y_i \, \exp\biggl(\int_{0}^{T} g_i(\theta, \Phi(\theta; 0, y)) \, \dd \theta \biggr),
\end{equation}
consequently $y_i > 0$.  If $x_i = 0$ we are done.  So assume that $x_i > 0$.  By part~(i), it suffices to exclude the case that $0 < x_i = y_i$ and $\Phi_i(\theta; 0, x) = \Phi_i(\theta; 0, y)$ for all $\theta \in [0, T]$.  We have then
\begin{equation*}
    \int_{0}^{T} \bigl( g_i(\theta, \Phi(\theta; 0, x)) - g_i(\theta, \Phi(\theta; 0, y)) \bigr) \, \dd \theta = 0.
\end{equation*}

But, from the proof of part~(i) it follows that $\Phi(\theta; 0, x) \le \Phi(\theta; 0, y)$.
Further, as $\Phi(T; 0, x) < \Phi(T; 0, y)$, there is $j \in \{1, \ldots, d\}$ such that $\Phi_j(T; 0, x) < \Phi_j(T; 0, y)$.  By part~(i), $\Phi_j(\theta; 0, x) < \Phi_j(\theta; 0, y)$ for all $\theta \in [0, T]$.  By~\ref{strongly-competitive}, $g_i(\theta, \Phi(\theta; 0, x)) \ge g_i(\theta, \Phi(\theta; 0, y))$ for a.a.\ $\theta \in (0, T)$, with strong inequality for $\theta$ in a set of positive measure.  The contradiction concludes the proof of part~(ii).

\end{proof}

\section{Analysis of the Poincaré map of (\ref{eq:ODE})}
\label{sect:Poincare}

From now on (except Appendix), our standing assumptions are
\begin{itemize}
    \item
    \ref{Cara},
    \item
    \ref{0-repel},
    \item
    \ref{new-H3} or \ref{tilde-H3},
    \item
    \ref{competitive}.
\end{itemize}
Recall that if~\ref{tilde-H3} is assumed the system is  normalized that $\hat{x}_i = 1$ for all $i \in \{1, \ldots, d\}$.

Sometimes, \ref{competitive} is strengthened to~\ref{strongly-competitive}.

\bigskip

In the following subsection  (\ref{subsect:CS}) we use our results obtained in \cite{B-M-JDEA} to prove the existence of a carrying simplex for the Poincaré map $\cP \colon C_+\to C_+$ defined by $\cP(x) := \Phi(T,0;x)$. In Section \ref{sect:time-periodic-CS} we will establish how to define a carrying simplex in extended phase space that incorporates the compact attractors of each map $\Phi(s,0;x)$ for $s\in (0,T)$.

\subsection{The carrying simplex of the Poincaré map of (\ref{eq:ODE})}
\label{subsect:CS}

We now use our recent results in \cite{B-M-JDEA} for retrotone maps to study the Poincar\'{e} map $P$ of (\ref{eq:ODE}).

\smallskip
\subsubsection{Retrotone and Weakly Retrotone maps}
We recall
\begin{definition}[Retrotone/Weakly retrotone]
A map $F\colon C_{+} \to C_{+}$ is \emph{weakly retrotone} in $B \subset C_{+}$ provided for all $I \subset \{1, \ldots, d\}$ and any $x, y \in B$, if
  \begin{equation*}
    F(x) < F(y) \quad \text{and} \quad F_i(x) < F_i(y) \text{ for all } i \in I
  \end{equation*}
  then
  \begin{equation*}
    x < y \quad \text{and} \quad x_i < y_i \text{ for all } i \in I.
  \end{equation*}

A map $F \colon C_{+} \to C_{+}$ is {\em retrotone} in a subset $B \subset C_{+}$, if, for all $x, y \in B$ with $ F(x) < F(y)$, one has that $x_i < y_i$ provided $y_i > 0$.
\end{definition}

In the above terminology Proposition~\ref{prop:retrotone} states that
\begin{itemize}
    \item
    for each $t>0$ the map $\Phi(t; 0, \cdot)$ is weakly retrotone in $C_{+}$;
    \item
    if we assume \ref{strongly-competitive}, the map $\Phi(T; 0, \cdot)$ is actually retrotone in $C_{+}$.
\end{itemize}

\smallskip
The main object of the investigation in our recent publication~\cite{B-M-JDEA} is a map $F \colon C_{+} \to C_{+}$, of Kolmogorov type, that is, $F := \diag [\mbox{id}] f$ where $f \colon C_+ \to C_+$.  To utilize the results of~\cite{B-M-JDEA} here we will use  the following set of assumptions (A):
\begin{enumerate}
\item[(A1)]
\quad $f$ is continuous, with $f(x) \gg 0$ for all $x \in C_{+}$;
\item[(A2)]
\quad  $f_i(\ev_i)= 1$,  $i=1,\ldots,d$;
\item[(A3)]
\quad
there exists $\varkappa > 0$ such that, putting $\Lambda := [0, (1+\varkappa) \ev]$,
    \begin{enumerate}
        \item[(A3-a)]
        \quad $F\!\!\restriction_{\Lambda}$ is a local homeomorphism, and
        \item[(A3-b)]
        \quad $F$ is weakly retrotone in $\Lambda$,
    \end{enumerate}
\item[(\textoverline{A4})]
\quad for any $x, y \in \Lambda$, if $F(x) < F(y)$ then
    \begin{enumerate}
        \item[(\textoverline{A4}-a)]
        \quad $f_i(x) \ge f_i(x)$ for all  $i$, and
        \item[(\textoverline{A4}-b)]
        \quad for those $i$ for which $F_i(x) < F_i(y)$ there holds either $x_i = 0$ or $f_i(x) > f_i(y)$.
    \end{enumerate}

\end{enumerate}

Sometimes the following stronger assumptions are made:
\begin{enumerate}
    \item[(A3\textprime)]
    \quad
    there exists $\varkappa > 0$ such that, putting $\Lambda := [0, (1+\varkappa) \ev]$,
    \begin{enumerate}
        \item[(A3\textprime-a)]
        \quad $F\!\!\restriction_{\Lambda}$ is a local homeomorphism, and
        \item[(A3\textprime-b)]
        \quad $F$ is retrotone in $\Lambda$,
    \end{enumerate}
    \item[(\textoverline{A4}\textprime)]
    \quad    for any $x, y \in \Lambda$, if $F(x) < F(y)$ then $f_i(x) > f_i(x)$ for all  $i$, provided $x_i > 0$.
\end{enumerate}

\subsection{Satisfaction of assumptions  (A) for the Poincaré map $P$}

In view of Lemma~\ref{lm:new:complete} $\cP(x)=\Phi(T;0,x)$ is well defined for all $x \in C_{+}$.

We write $\cP^n$ for the $n$\nobreakdash-\hspace{0pt}th iterate of $\cP$, with $\cP^0 = \Id_{C_{+}}$.
Eqs.~\eqref{eq:process} and~\eqref{eq:process-periodic} imply that
\begin{multline*}
    \Phi((n + 1) T ; 0, x) = \Phi((n + 1) T ; n T, \Phi(n T; 0, x))
    \\
    = \Phi(T; 0, \Phi(n T; 0, x)) = \cP(\Phi(n T; 0, x)), \quad x \in C_{+}, \ n = 0, 1, 2, \ldots,
\end{multline*}
which gives $\cP^{n + 1} = \cP \circ \cP^{n}$ and $\cP^n = \Phi(n T; 0, \cdot)$.  Hence $(\cP^n)_{n = 0}^{\infty}$ is a discrete-time (semi)dynamical system on $C_{+}$.

\smallskip
 There holds $\cP(x_i \ev_i) = \hat{H}_i(x_i) \e_i$ for $x_{i} \in [0, \infty)$.

\begin{proposition}
    Let $\varkappa > 0$ be arbitrary.  Then the assumptions \textup{(A1)}, \textup{(A2)}, \textup{(A3)} and \textup{(\textoverline{A4})} are satisfied.  Under~\ref{strongly-competitive} the assumptions \textup{(A3)\textprime} and \textup{(\textoverline{A4}\textprime)} are satisfied, too.
\end{proposition}
\begin{proof}
    For $i \in \{1, \ldots, d\}$ we have
    \begin{equation*}
        \frac{P_i(x)}{x_i}   =    \frac{\Phi_i(T; 0, x)}{x_i} = \exp{\Biggl(\int_{0}^{T} g_i(\tau, \Phi(\tau; 0, x)) \, \mathrm{d}\tau \Biggr)}
    \end{equation*}

    as long as $x \in C_{+} \setminus (C_{\{1, \ldots, d\} \setminus \{i\}})_{+}$.  Since the latter is dense in $C_{+}$, Lemma~\ref{lm:new:complete}, Theorem~\ref{thm:cont-depend} and Lemma~\ref{lm:continuity-integrals} allow us to define a continuous $f_i \colon C_{+} \to C_{+}$ by the formula
    \begin{equation}
    \label{eq:f_i}
        f_i(x) := \exp{\Biggl(\int_{0}^{T} g_i(\tau, \Phi(\tau; 0, x)) \, \mathrm{d}\tau \Biggr)}.
    \end{equation}
    $f_i(x)$ is clearly positive for all $x \in C_{+}$ and all $1 \le i \le d$.  Consequently, $\cP$ is of Kolmogorov type with $f$ continuous, and $f(x) \gg 0$ for all $x \in C_{+}$, so (A1) is satisfied.

    The satisfaction of (A2) is a consequence of Proposition~\ref{prop:axes}(i) and the normalization.

    (A3-a) is satisfied by Lemma~\ref{lm:homeo}.  The satisfaction of (A3-b) (or (A3\textprime-b) under~\ref{strongly-competitive}) is a consequence of Proposition~\ref{prop:retrotone}.

    Let $\cP(x) < \cP(y)$.  To prove (\textoverline{A4})/(\textoverline{A4}\textprime) it suffices, by~\eqref{eq:f_i}, to check the sign of
    \begin{equation}
    \label{eq:integral}
        \int_{0}^{T} \bigl(g_i(\tau, \Phi(\tau; 0, x)) - g_i(\tau, \Phi(\tau; 0, y)\bigr) \, \mathrm{d}\tau.
    \end{equation}
    \eqref{eq:integral} is non-positive for all $i$, by the first part of~\ref{competitive} and Proposition~\ref{prop:retrotone}, and negative for those $i$ for which $\cP_i(x) < \cP_i(y)$, by the second part of~\ref{competitive} and Proposition~\ref{prop:retrotone}(i).  Under ~\ref{strongly-competitive}, \eqref{eq:integral} is negative for all $i$ by Proposition~\ref{prop:retrotone}(ii).
\end{proof}

\medskip
Emulating Subsection~2.3 of~\cite{B-M-JDEA}, we define for the Poincar\'{e} map $P$
\begin{equation*}
    \Gamma := \bigcap_{n = 0}^{\infty} \cP^n([0, (1 + \varkappa)\ev]),
\end{equation*}
where $\varkappa > 0$.

There holds
\begin{equation*}
    \Gamma = \bigcap_{n = 0}^{\infty} \cP^n([0, \ev]),
\end{equation*}
and $[0, (1 + \varkappa)\ev]$ attracts bounded sets in $C_{+}$ (see~\cite[Lemmas~2.9 and~2.10]{B-M-JDEA}).  Then we have
\begin{theorem}
\label{thm:global-attractor-exists}
    For the dynamical system $(\cP^n)_{n = 0}^{\infty}$
    \begin{enumerate}
        \item[\textnormal{(i)}]
        $\Gamma$ is the compact attractor of bounded sets in $C_{+}$,
        \item[\textnormal{(ii)}]
        $\Gamma$ is characterized as the set of those $x \in C_{+}$ for which there exists a bounded total $\cP$\nobreakdash-\hspace{0pt}orbit.
    \end{enumerate}
\end{theorem}

\section{Construction of the global attractor of $\Phi$ via the Poincaré map $P$}
We introduce now another local dynamical system, on this occasion with continuous time, generated by~\eqref{eq:ODE} and denoted by $\mySemi$.

Let $\mytorus$ stand for the interval $[0, T]$ with endpoints identified.  In other words, $a,b \in \RR$ correspond to the same element of $\mytorus$ if and only if $a \equiv b \hspace{5pt} (\hspace{-7pt}\mod{T})$.  The sum $t + s$ where $t \in \RR$ and $s \in \mytorus$ is considered modulo $T$.  In calculations we sometimes identify $\mytorus$ with $[0, T)$.

The distance between $s_1, s_2 \in \mytorus$ will be denoted by
\begin{equation*}
    d_{\mytorus}(s_1, s_2) := \min \{\abs{s^\dagger_1 - s^\dagger_2}, T - \abs{s^\dagger_1 - s^\dagger_2}\},
\end{equation*}
where $s^{\dagger}_j$, $j = 1, 2$, stands for the (unique) representative of $s_j$ in $[0, T)$.

We define $\mySemi \colon [0, \infty) \times (\mytorus \times C_{+}) \to \mytorus \times C_{+}$ by the formula
\begin{equation*}
    \mySemi(t, (s, x)) := \overline{\Phi}(t, (s^\dagger, x)) = (s^\dagger + t, \Phi(s^\dagger + t; s^\dagger, x))
\end{equation*}
where $s^\dagger \in \RR$ is \textit{any} representative of $s \in \mytorus$. We refer to $\mytorus \times C_+$ as the \emph{extended phase space}. \eqref{eq:process-periodic} guarantees that $\mySemi$ is well-defined.  Further, we have
\begin{itemize}
    \item[(S0)]
        $\mySemi$ is continuous;
    \item[(S1)]
        $\mySemi(0, (s, x)) = (s, x)$ for all $(s,x) \in \mytorus \times C_{+}$;
    \item[(S2)] For all $(s,x)\in \mytorus \times C_+$
        \begin{equation*}
            \mySemi(t_2, \mySemi(t_1, (s,x)) = \mySemi(t_1 + t_2, (s, x)), \quad t_1, t_2 \ge 0.
        \end{equation*}
\end{itemize}
By the remark below~\eqref{eq:process-inverse}, $\mySemi(t, \cdot)$ is injective for all $t \ge 0$.

Observe that by (S2) we have
\begin{equation*}
    \mySemi(n T, (0, x)) = (0, \cP^n(x)), \quad n \in \NN, \ x \in C_{+}.
\end{equation*}

In the following we use a `tilde' over symbols to indicate membership of the extended phase space $\mytorus \times C_+$.

\medskip

For $\tilde{x}=(s_1, x^{(1)}), \ \tilde{y} = (s_2, x^{(2)}) \in \mytorus \times C_{+}$ we introduce the metric $\tilde{d}$  on the extended phase space by
\begin{equation*}
    \tilde{d}(\tilde{x},\tilde{y}) := d_{\mytorus}(s_1, s_2) + \norm{x^{(1)} - x^{(2)}},
\end{equation*}
and for $\tilde{x} \in \mytorus \times C_{+}$ and $\widetilde{A} \subset \mytorus \times C_{+}$ we write
\begin{equation*}
    \tildedist(\tilde{x}, \widetilde{A}) := \inf\{\,\tilde{d}(\tilde{x}, \tilde{y}) :\tilde{y} \in \widetilde{A} \,\}.
\end{equation*}
For $\widetilde{A} \in \mytorus \times C_{+}$ and $s \in \mytorus$ we denote by $\widetilde{A}_s$ the \emph{$s$\nobreakdash-\hspace{0pt}section} of $\widetilde{A}$, $\widetilde{A}_s := \{\, x \in C_{+} : (s, x) \in \widetilde{A} \,\}$,  so that $\widetilde{A}$ can be decomposed
\begin{equation*}
    \widetilde{A} = \bigcup_{s \in \mytorus} \bigl( \{s\} \times \widetilde{A}_s \bigr).
\end{equation*}

For $D \subset C_{+}$ denote by $\widetilde{D}$ the set
\begin{equation*}
    \{\, \mySemi(s, (0, x)): s \in \mytorus,\ x \in D \,\}.
\end{equation*}
Then the $s$\nobreakdash-\hspace{0pt}section $\widetilde{D}_s$ equals $\{\, \Phi(s; 0, x) : x \in D \,\}$.

Observe that if $D \subset C_{+}$ is compact, $\widetilde{D}$, being the image of the compact set $[0, T] \times \{0\} \times D$ under the continuous map $\mathbb{S}$, is compact, too.

\begin{proposition}
\label{prop:sections-continuity}
    Let $D \subset C_{+}$ be compact.  Then the following holds.
    \begin{enumerate}[label=\textup{(\roman*)}, ref=\textup{\ref{prop:sections-continuity}}\textup{(\roman*)}]
        \item\label{prop:sections-continuity-i}
        The assignment
            \begin{equation*}
                [\, \mytorus \ni s \mapsto \widetilde{D}_s \in \mathcal{P}(C_{+}) \,]
            \end{equation*}
        is continuous;
        \item\label{prop:sections-continuity-ii}
        $\bigcup\limits_{s \in \mytorus} \widetilde{D}_s$
        is compact.
    \end{enumerate}
\end{proposition}
\begin{proof}
    Fix $s_m \to s$ in $\mytorus$.

    Let $(x^{(m_l)})_{l = 1}^{\infty}$ be a subsequence such that $x^{(m_l)} \in \widetilde{D}_{s_{m_l}}$ and $\lim\limits_{l \to \infty} x^{(m_l)} = x$.  There exists a unique sequence $\xi^{(m_l)} \in {D}$ such that $x^{(m_l)} = \Phi(s_{m_l}; 0, \xi^{(m_l)})$.  By compactness, there is a subsequence $(\xi^{(m_{l_p})})_{p =1}^{\infty}$ convergent, as $p \to \infty$, to some $\xi \in {D}$.  Theorem~\ref{thm:cont-depend} implies that $\Phi(s_{m_{l_p}}; 0, \xi^{(m_{l_p})})$ converges, as $p \to \infty$, to $\Phi(s; 0, \xi) = x$.  Therefore $x \in \widetilde{D}_s$, so (PK2) in the definition of Kuratowski convergence is satisfied.

    On the other hand, if $x \in \widetilde{D}_s$ then there is a (unique) $\xi \in {D}$ such that $x = \Phi(s; 0, \xi)$.  Again by Theorem~\ref{thm:cont-depend}, $x$ is the limit of the sequence $\Phi(s_{m_l}; 0 , \xi) \in \widetilde{D}_{s_{m_l}}$, so (PK1) is satisfied.

    Regarding part (ii), observe that the union in~question is the image of the compact set $[0, T] \times \{0\} \times D$ under a continuous map, namely $\mathbb{S}$ composed with the projection on the $C_{+}$\nobreakdash-\hspace{0pt}axis.
\end{proof}

We say that $\widetilde{A} \subset \mytorus \times C_{+}$ is \emph{forward $\mySemi$\nobreakdash-\hspace{0pt}invariant} if $\mySemi(t, \widetilde{A}) \subset \widetilde{A}$ for all $t \ge 0$, and \emph{$\mySemi$\nobreakdash-\hspace{0pt}invariant} if $\mySemi(t, \widetilde{A}) = \widetilde{A}$ for all $t \ge 0$.

For $\tilde{x} \in \mytorus \times C_{+}$, a set $\{\, \tilde{x}^{(\theta)} : \theta \in (- \infty, 0] \,\}$ of elements of $\mytorus \times C_{+}$ is a \emph{backward $\mySemi$\nobreakdash-\hspace{0pt}orbit of $x$} if $\tilde{x}^{(0)} = \tilde{x}$ and $\mySemi(t, \tilde{x}^{(\theta)}) = \tilde{x}^{(\theta + t)}$ for $\theta \le 0$ and $t \ge 0$ with $\theta + t \le 0$.  By the injectivity of $\mySemi(t, \cdot)$, a backward $\mySemi$\nobreakdash-\hspace{0pt}orbit of $\tilde{x}$, if it exists, is unique.  In such a case we write $\mySemi(- t, \tilde{x})$ rather than $\tilde{x}^{(-t)}$, $t \ge 0$.

The \emph{forward $\mySemi$\nobreakdash-\hspace{0pt}orbit of $\tilde{x}$} equals $\{\, \mySemi(t, \tilde{x}) : t \ge 0 \,\}$.  If a backward $\mySemi$\nobreakdash-\hspace{0pt}orbit of $\tilde{x}$ exists, we call the union of the backward and forward $\mySemi$\nobreakdash-\hspace{0pt}orbits the \emph{total $\mySemi$\nobreakdash-\hspace{0pt}orbit of $\tilde{x}$}.  The total $\mySemi$\nobreakdash-\hspace{0pt}orbit, if it exists, is unique.

\begin{lemma}
\label{lm:bar-D-invariant}
    If $D \subset C_{+}$ is forward $\cP$\nobreakdash-\hspace{0pt}invariant, $\widetilde{D}$ is forward $\mySemi$\nobreakdash-\hspace{0pt}invariant.  If $D \subset C_{+}$ is $\cP$\nobreakdash-\hspace{0pt}invariant, $\widetilde{D}$ is $\mySemi$\nobreakdash-\hspace{0pt}invariant.
\end{lemma}
\begin{proof}
    Let $t \ge 0$ and $\mySemi(s, (0, x)) \in \widetilde{D}$ with $s \in \mytorus$, $x \in D$.   Then $\mySemi(t, \mySemi(s, (0, x))) = \mySemi(t + s, (0, x)) = \mySemi(t + s - m T, \mySemi(mT, (0, x))$ with $m := \lfloor \frac{t + s}{T} \rfloor$.  But $t + s - m T \in \mytorus$ and $\mySemi(mT, (0, x)) = (0, \cP^m(x))$, with $\cP^m(x) \in D$.  This proves that $\mySemi(t, \widetilde{D}) \subset \widetilde{D}$ for any $t \ge 0$

    Assume $D$ to be $\cP$\nobreakdash-\hspace{0pt}invariant. Putting $n := - \lfloor \frac{s - t}{T} \rfloor$ and $y := (\cP^n\hspace{-.375em}\restriction_{D})^{-1}(x) \in D$ we obtain $\mySemi(t, \mySemi(s - t + nT, (0, y))) = \mySemi(s, \mySemi(nT,(0, y))) = \mySemi(s, (0, x))$ with $s - t + nT \in \mytorus$.
\end{proof}
In particular, for a $\cP$\nobreakdash-\hspace{0pt}invariant $D$ there holds
\begin{equation}
\label{eq:aux3}
    \Phi(t; 0, D) = \Phi(t - \lfloor \tfrac{t}{T} \rfloor T; 0, D), \quad t \ge 0.
\end{equation}

\smallskip
Let $\mathcal{\widetilde{A}}$ be a class of subsets of $\mytorus \times C_{+}$.  We say that a nonempty compact $\mySemi$\nobreakdash-\hspace{0pt}invariant $\widetilde{D} \subset \mytorus \times C_{+}$ is a \emph{compact attractor for the dynamical system $\mySemi$, of the class $\mathcal{\widetilde{A}}$} if for any $\widetilde{A} \in \mathcal{\widetilde{A}}$ and any $\epsilon > 0$ there exists $t_0 = t_0(\widetilde{A}, \epsilon)$ such that $\tildedist(\mySemi(t, \tilde{x}), \widetilde{D}) < \epsilon$ for all $\tilde{x} \in \widetilde{A}$ and all $t \ge t_0$.

\begin{lemma}
\label{lm:attractor}
    Let $D \subset C_{+}$ be a compact attractor  for the dynamical system $(\cP^n)_{n = 0}^{\infty}$ of a class $\mathcal{A}$ of subsets of $C_{+}$.  Then for any $A \in \mathcal{A}$ and any $\epsilon > 0$ there exists $t_0 = t_0(A, \epsilon)$ such that for all $x \in A$ and all $t \ge t_0$ there holds $\dist(\Phi(t; 0, x), \Phi(t; 0, D)) < \epsilon$.
\end{lemma}
\begin{proof}
    Fix $\epsilon > 0$.  By the continuous dependence of solutions on the initial values and the compactness of $D \times [0, T]$, for each $y \in D$ there is $\delta_y > 0$ such that if $\xi \in C_{+}$, $\norm{\xi - y} < \delta_y$ then $\norm{\Phi(s; 0, \xi) - \Phi(s; 0, y)} < \epsilon$ for all $s \in [0, T]$.  Since $D$ is compact, we can extract a finite subcover $\{\, B(y_1; \delta_{y_1}), \ldots, B(y_m; \delta_{y_m}) \,\}$ of the open cover $\{\, B(y; \delta_{y}) : y \in D \,\}$.  Let $\tilde{\delta} > 0$ be so small that
    \begin{equation*}
        \{\, \eta \in C_{+} : \dist(\eta, D) < \tilde{\delta} \,\} \subset \bigcup_{k = 1}^{m} B(y_k; \delta_{y_k}).
    \end{equation*}
    As $D$ is, for $(\cP^n)_{n = 0}^{\infty}$, an attractor of sets in $\mathcal{A}$, for any $A \in \mathcal{A}$ there is $n_0 \in \NN$ such that $\dist(\Phi(n T; 0, x), D) = \dist(\cP^n(x), D) < \tilde{\delta}$ for all $n \ge n_0$ and all $x \in A$.

    Fix $A \in \mathcal{A}$, and put $t_0 := n_0 T$.  Any $t \ge t_0$ can be written as $t = n T + \theta$ with $n \ge n_0$ and $\theta \in [0, T)$.  For any $x \in A$ we have $\Phi(n T; 0, x) \in B(y_k; \delta_{y_k})$ for some $1 \le k \le m$.  By~\eqref{eq:process} and~\eqref{eq:process-periodic}, $\Phi(n T + \theta; 0, x) = \Phi(n T + \theta; n T, \Phi(n T; 0, x)) = \Phi(\theta; 0, \Phi(n T; 0, x))$, so $\dist(\Phi(t; 0, x), \widetilde{D}_{\theta}) \le \norm{\Phi(t; 0, x) - \Phi(\theta; 0, y_k)} < \epsilon$.  By~\eqref{eq:aux3}, $\widetilde{D}_{\theta} = \Phi(t; 0, D)$.
\end{proof}
As, for any $s \in \mytorus$ and any $\xi \in C_{+}$, $\tildedist((s, x), \widetilde{D}) \le \dist(\xi, \widetilde{D}_s)$ we have the following.
\begin{corollary}
\label{cor:attractor}
    Let $D \subset C_{+}$ be a compact attractor for the dynamical system $(\cP^n)_{n = 0}^{\infty}$ of a class $\mathcal{A}$ of subsets of $C_{+}$.  Then $\widetilde{D}$ is a compact attractor for $\mySemi$ of the class $\{\, \{0\} \times A : A \in \mathcal{A} \,\}$.
\end{corollary}

\begin{theorem}
\label{thm:attractor-ODE}
    \begin{enumerate}[label=\textup{(\roman*)}, ref=\textup{\ref{thm:attractor-ODE}}\textup{(\roman*)}]
        \item
        \label{thm:attractor-ODE-i}
        For any bounded $A \subset C_{+}$ and any $\epsilon > 0$ there exists $n_0 \in \NN$ such that $\dist(\Phi(t; 0, x), \Phi(t; 0, \Gamma)) < \epsilon$ for all $x \in A$ and all $t > n_0 T$.
        \item
        \label{thm:attractor-ODE-ii}
        $\Gamma$ is characterized as the set of those $x \in C_{+}$ for which the solution of~\eqref{eq:ODE} passing through $x$ at $t = 0$ is defined and bounded on $(-\infty, \infty)$.
    \end{enumerate}
\end{theorem}
\begin{proof}
    (i) is a consequence of Theorem~\ref{thm:global-attractor-exists}(i) and Lemma~\ref{lm:attractor}.

    Assume $x \in \Gamma$.  By Theorem~\ref{thm:global-attractor-exists}(ii), $\cP^n(x)$ exist for all $n \in \ZZ$, so we have $\ZZ \subset (\taumin(0, x), \taumax(0, x))$, consequently $(\taumin(0, x), \taumax(0, x)) = (-\infty, \infty)$.  Since $\Gamma$ is $\cP$\nobreakdash-\hspace{0pt}invariant, there holds $\{\, \Phi(t; 0, \Gamma) : t \in (-\infty, \infty) \,\} = \{\, \Phi(t; 0, \Gamma) : t \in [0, T] \,\}$.  The latter is compact, so the solution passing through $x$ at $t = 0$ is bounded.

    If $x \in C_{+}$ is such that the solution of~\eqref{eq:ODE} passing through it at $t = 0$ is defined and bounded on $(-\infty, \infty)$, then $\cP^n(x)$ exist for all $n \in \ZZ$ and the set $\{\, \cP^n(x) : n \in \ZZ \,\}$ is bounded, so, by Theorem~\ref{thm:global-attractor-exists}(ii), $x \in \Gamma$.
\end{proof}

Lemma~\ref{lm:bar-D-invariant} applied to $\Gamma$ allows us to extend $\mySemi$ to a flow on $\widetilde{\Gamma}$: for $t < 0$ and $(s, x) \in \widetilde{\Gamma}$ we define $\mySemi(t, (s, x))$ to be equal to $(s', y)$ such that $\mySemi(-t, (s', y))$.  (S2) is satisfied for any $t_1, t_2 \in \RR$ and any $(s, x) \in \widetilde{\Gamma}$:
\begin{equation}
\label{eq:flow-on-Gamma}
    \mySemi(t_2, \mySemi(t_1, (s,x)) = \mySemi(t_1 + t_2, (s, x)), \quad t_1, t_2 \in \RR, \ x \in \widetilde{\Gamma}.
\end{equation}
Similarly, the process $\Phi$ restricted to $\Gamma$ is defined for any two $s, t \in \RR$, and satisfies
\begin{equation}
\label{eq:process-on-Gamma}
     \Phi(t_1 + t_2; t_1, \Phi(t_1; 0, x))) = \Phi(t_1 + t_2; 0, x), \quad t_1, t_2 \in \RR, \ x \in \Gamma.
\end{equation}
\begin{theorem}~\
\label{thm:bar-S-attractor}
       $\widetilde{\Gamma}$ is the compact attractor, for the dynamical system $\mySemi$, of bounded sets in $\mytorus \times C_{+}$.
\end{theorem}
\begin{proof}
    Let $\widetilde{A} \subset \mytorus \times C_{+}$ be bounded.  Put $B := \bigcup\limits_{s \in [0, 1)} \Phi(T; s, \widetilde{A}_s)$.  The set $B$ is bounded, moreover for any $(s, x) \in \widetilde{A}$ there holds $\mySemi(T - s, (s, x)) = (0, \Phi(T; s, x))$ with $\Phi(T; s, x) \in B$.  By Corollary~\ref{cor:attractor}, for any $\epsilon > 0$ there is $t_0$ such that $\tildedist(\mySemi(t, (0,y)), \widetilde{\Gamma}) < \epsilon$ for all $y \in B$ and all $t \ge t_0$.  By the definition of $B$, for any $(s, x) \in \widetilde{A}$ there holds $\mySemi(T - s, (s, x)) = (0, \Phi(T; s, x))$ with $\Phi(T; s, x) \in B$.  For any $t > T$ we have
    \begin{equation*}
        \mySemi(t - s, (s, x)) = \mySemi(t - T, \mySemi(T - s, (s, x))) = \mySemi(t - T, (0, \Phi(T; s, x))).
    \end{equation*}
    Consequently, for all $t \ge t_0 + T$ and all $(s, x) \in \widetilde{A}$ we have $\tildedist(\mySemi(t, (s,x)), \widetilde{\Gamma}) < \epsilon$.
\end{proof}

\section{The time-periodic carrying simplex}
\label{sect:time-periodic-CS}
We pause a~little to formulate main results for carrying simplices of retrotone maps in~\cite{B-M-JDEA}.  It was proved there that under assumptions (A1)--(A3) and (\textoverline{A4}) [resp.\ under assumptions (A1)--(A2), (A3)\textprime\ and (\textoverline{A4}\textprime)] there exists, for the Kolmogorov map $F \colon C_{+} \to C_{+}$, a set $\Sigma \subset C_{+}$ with the following properties:
\begin{enumerate}[label=\textup{(CS\arabic*)},ref=\textup{(CS\arabic*)},leftmargin=50pt]
\item
\label{P-unordered}
    $\Sigma$ is an unordered \textup{[}resp.\ weakly unordered\textup{]} subset of $\Gamma$.
\item
\label{P-radial-proj}
    $\Sigma$ is homeomorphic via radial projection to the $(d-1)$\nobreakdash-\hspace{0pt}dimensional standard probability simplex $\Delta$.
\item
\label{P-invariant}
    $F(\Sigma) = \Sigma$ and $F{\restriction}_{\Sigma} \colon \Sigma \to \Sigma$ is a homeomorphism.
\item
\label{P-attraction}
    $\Sigma$ is the compact attractor, for $(F^n)_{n = 0}^{\infty}$, of bounded $A \subset C_{+}$ with $0 \notin \bar{A}$.
\item
\label{P-asymptotic}
    For any $x \in C_{+} \setminus \{0\}$ there is $y \in \Sigma$ such that $\displaystyle \lim\limits_{n\to +\infty} \norm{F^n(x) - F^n(y)} = 0$ \textup{(}this property is called \emph{asymptotic completeness} or \emph{asymptotic phase}\textup{)} \textup{[}resp.\ for any $x \in \Gamma \setminus \{0\}$ there is $y \in \Sigma$ such that $\displaystyle \lim\limits_{n\to +\infty} \norm{F^n(x) - F^n(y)} = 0$\textup{]}.
    \item
    \label{P-order-convex}
    $\Sigma$ is the boundary \textup{(}relative to $C_{+}$\textup{)} of $\Gamma$ and $\Gamma$ is order convex.  In~particular, $\Gamma = \{\alpha x :  \alpha \in [0, 1], \ x \in \Sigma\}$.
    \item
    \label{P-attractor}
    $\Gamma \setminus \Sigma = \{\alpha x :  \alpha\in [0, 1), x \in \Sigma\}$ is characterised as the set of all those $x \in C_{+}$ that have a backward $F$\nobreakdash-\hspace{0pt}orbit $\{\ldots, x^{(-2)}, x^{(-1)}, x\}$ with $\lim\limits_{n \to \infty} x^{(-n)} = 0$.
    \item
    \label{P-characterization}
    $\Sigma$ is characterised as the set of all $x \in C_{+}$ having total $F$\nobreakdash-\hspace{0pt}orbits that are bounded and bounded away from $0$.
    \item
    \label{P-Lipschitz}
    The inverse $(\Pi{\restriction}_{\Sigma})^{-1}$ of the orthogonal projection of $\Sigma$ along $\ev$ is Lipschitz continuous.
\end{enumerate}

\begin{definition}[Weak carrying simplex/carrying simplex]
A set $\Sigma$ satisfying \ref{P-unordered}--\ref{P-asymptotic} is referred to as the \emph{weak carrying simplex} for $F$.  Under \textup{(A3)\textprime}\ and \textup{(\textoverline{A4}\textprime)}, $\Sigma$, satisfying the strengthened versions of~\ref{P-unordered} and~\ref{P-unordered}, is called the \emph{carrying simplex} for $F$.
\end{definition}

In~\cite{B-M-JDEA} a quite novel construction of the (weak) carrying simplex is given, as the limit of the iterates under $F$ of a homothetic image of the probability  simplex $\Delta$.  To be more specific, we have
\begin{theorem}
\label{thm:construction-0}
    Assume \textup{(A1)--(A3)} and \textup{(\textoverline{A4})}.  Then there exists a weak carrying simplex \textup{(}under \textup{(A3)}\textprime\ and \textup{(\textoverline{A4}\textprime)} a carrying simplex\textup{)} $\Sigma$ for $F$, and the following holds:
    \begin{enumerate}
        \item[\textup{(i)}]
        There exists $\varepsilon > 0$ such that $F^n(\varepsilon \Delta) \subset \Lambda$ for all $n \in \NN$, and
        \begin{equation*}
            \Sigma = \lim\limits_{n \to \infty} F^n(\varepsilon \Delta),
        \end{equation*}
        \item[\textup{(ii)}]
        $F^n(H) \subset \Lambda$ for all $n \in \NN$, and
        \begin{equation*}
            \Sigma = \lim\limits_{n \to \infty} F^n(H),
        \end{equation*}
        where
        \begin{equation*}
            H := \{\, x \in C_{+} : \forall i \in \{1, \ldots, d\} \ x_i \le 1 + \varkappa, \ \exists j \in \{1, \ldots, d\} \ x_j = 1 + \varkappa \,\}
        \end{equation*}
        with $\varkappa > 0$ as in \textup{(A3)} \textup{(}or as in \textup{(A3\textprime))}.
    \end{enumerate}
\end{theorem}
By \cite[Sect.~3.3]{B-M-JDEA}, for each $n \in \NN$ we can write
\begin{equation*}
    F^n(\varepsilon \Delta) = \{\, R_{n T}(u) u : u \in \Delta \,\}, \quad F^n(H) = \{\, R^{n T}(u) u : u \in \Delta \,\},
\end{equation*}
where $R_{n T} \colon \Delta \to (0, \infty)$ [$R^{n T} \colon \Delta \to (0, \infty)$] are continuous functions (called the \emph{radial representations} of $F^n(\varepsilon \Delta)$ [$F^n(H)$].  Further, there is a continuous function $R^{*} \colon \Delta \to (0, \infty)$ (called the \emph{radial representation} of $\Sigma$) such that
\begin{equation*}
    \Sigma = \{\, R^{*}(u) u : u \in \Delta \,\}.
\end{equation*}
In the language of radial representations we have the following, see~\cite[Prop.~3.3(a)]{B-M-JDEA}:
\begin{proposition}  \label{prop:construction-0}
    Assume \textup{(A1)--(A3)} and \textup{(\textoverline{A4})}.  Then
    \begin{enumerate}
        \item[\textup{(i)}]
        $R_{n T}$ converge uniformly, as $n \to \infty$, to $R^{*}$; and
        \item[\textup{(ii)}]
        $R^{n T}$ converge uniformly, as $n \to \infty$, to $R^{*}$.
    \end{enumerate}
\end{proposition}

\medskip
We go back to our reasoning.  Recall that the standing assumptions are that \ref{Cara}, \ref{0-repel}, \ref{new-H3} or \ref{tilde-H3}, and (at~least) \ref{competitive} are satisfied.
\begin{theorem}~\
\label{thm:tilde-Sigma-attractor}
The following holds:
      \begin{enumerate}[label=\textup{(\roman*)}, ref=\textup{\ref{thm:tilde-Sigma-attractor}}\textup{(\roman*)}]
        \item
        \label{thm:tilde-Sigma-attractor-attracts}
        $\widetilde{\Sigma} \subset \widetilde{\Gamma}$ is the compact attractor, for the dynamical system $\mySemi$, of bounded sets in $\mytorus \times C_{+}$ whose closures are disjoint with $\mytorus \times \{0\}$;
        \item
        \label{thm:tilde-Sigma-attractor-phase}
        For any $x \in \Gamma \setminus \{0\}$ there is $y \in \Sigma$ such that $\displaystyle \lim\limits_{t \to +\infty} \norm{\Phi(t;0,x) -\Phi(t;0,y)} = 0$ \textup{[}under~\ref{strongly-competitive}, for any $x \in C_{+} \setminus \{0\}$ there is $y \in \Sigma$ such that $\displaystyle \lim\limits_{t \to +\infty} \norm{\Phi(t;0,x) -\Phi(t;0,y)} = 0$\textup{]};
        \item
        \label{thm:tilde-Sigma-attractor-char_of_differ}
        $\Gamma \setminus \Sigma$ is characterized as the set of those  $x \in C_{+}$ for which $\lim\limits_{t \to - \infty} \Phi(t;0,x) = 0$;
        \item
        \label{thm:tilde-Sigma-attractor-characterization}
        $\Sigma$ is characterised as the set of those $x \in C_{+}$ for which $\{\, \Phi(t; 0, x) : t \in (- \infty,\infty) \,\}$ is bounded and bounded away from $0$;
        \item
        \label{thm:tilde-Sigma-attractor-order-convex}
        $\Sigma$ is the boundary \textup{(}relative to $C_{+}$\textup{)} of $\Gamma$ and $\Gamma$ is order convex; in~particular, $\Gamma = \{\alpha x :  \alpha \in [0, 1], \ x \in \Sigma\}$.
    \end{enumerate}
\end{theorem}
\begin{proof}
    We show first that $\widetilde{\Sigma}$ is the compact attractor, for $\mySemi$, of compact sets in $\mytorus \times C_{+}$ that are disjoint with $\mytorus \times \{0\}$.  Let $\widetilde{\mathcal{A}}$ stand for the class of such sets.  Take $\widetilde{A} \in \widetilde{\mathcal{A}}$, and put
    \begin{equation*}
        B := \widetilde{A}_{0} \cup \bigcup\limits_{s \in [0, T)} \Phi(T; s , \widetilde{A}_s).
    \end{equation*}
    The set $B$ is clearly bounded.  We show that $B$ is closed.  Let $(y^{({m})})_{m = 1}^{\infty} \subset B$ be a sequence convergent to some $y$.  If $y \in \widetilde{A}_0$ we are done.  Assume that $y \notin \widetilde{A}_0$.  Since $\widetilde{A}_0$ is closed, after possibly ignoring finitely many terms we can write $y^{({m})} = \Phi(T; s_{m}, x^{({m})})$ with $(s_{m}, x^{({m})}) \in \widetilde{A}$. Because $\widetilde{A}$ is compact, we can extract a subsequence $(s_{m_k}, x^{({m}_k)})$ convergent to some $(s_0, x) \in \widetilde{A}$ as $k \to \infty$.  The case $s_0 = T$ is impossible, since then $y = \Phi(T; T, x) = \Phi(0; 0, x) = x$ and $(0, x) = (0, y) \in \widetilde{A}$, that is, $y \in \widetilde{A}_0$.  Therefore, there exists $\delta > 0$ such that $y^{({m}_k)} = \Phi(T; s_{m_k}, x^{({m}_k)})$ with $(s_{m_k}, x^{({m}_k)}) \in \widetilde{A}$ and $T - s_{m_k} \ge \delta$.  Hence, $y = \Phi(T; s_0, x) \in \Phi(T; s_0, \widetilde{A}_{s_0})$.  It follows from the uniqueness of solutions that $0 \notin B$.  We now apply Corollary~\ref{cor:attractor} as in the proof of Theorem~\ref{thm:bar-S-attractor} to conclude that $\widetilde{\Sigma}$ is the compact attractor of compact sets in $\mytorus \times C_{+}$ that are disjoint with $\mytorus \times \{0\}$.  As such, it is contained in the compact attractor $\widetilde{\Gamma}$ of bounded subsets of $\mytorus \times C_{+}$.

    Let $x$ and $y$ be as in (ii).  The set $\{\, \cP^n(x) : n \in \NN \,\} \cup \Sigma$ is bounded and closed (indeed, its accumulation points belong to $\Sigma$), hence compact.  Take $\epsilon > 0$. Copying the reasoning in the proof of Lemma~\ref{lm:attractor} we obtain the existence of $\delta > 0$ such that if $\norm{\xi - \cP^n(x)} < \delta$ then $\norm{\Phi(s; 0, \xi) - \Phi(s; 0, \cP^n(x))} < \epsilon/2$ for all $n \in \NN$ and all $s \in [0, T]$.  Let $n_0 \in \NN$ be such that $\norm{\cP^n(x) - \cP^n(y)} < \min\{\delta, \epsilon/2\}$ for all $n \ge n_0$.  Representing $t \ge n_0$ as $t = n T + s$, $s \in [0, T)$, we estimate
    \begin{equation*}
        \norm{\Phi(t;0,x) - \Phi(t;0,y)} \le \norm{\Phi(s;0,\cP^n(x)) - \Phi(s;0,\cP^n(y))} + \norm{\cP^n(x) - \cP^n(y)} < \epsilon,
    \end{equation*}
    which concludes the proof of part~(ii).

    In view of~\ref{P-attractor} the only thing needed to prove in~(iii) is that if $x \in C_{+}$ has the property that $\Phi(t; 0, x)$ exists for all $t \le 0$ and the set $\lim\limits_{n \to \infty} \Phi(-n; 0, x) = 0$ then $\lim\limits_{t \to - \infty} \Phi(t; 0, x) = 0$.  As the set $\{\, \Phi(-n; 0, x) : n \in \NN \,\} \cup \{0\}$ is compact we can copy the proof in the previous paragraph.

    Observe that, by uniqueness, the projection $\widehat{\Sigma}$ of $\widetilde{\Sigma}$ onto the second axis is a compact subset of $C_{+}$ not containing $0$.  As, for any $x \in \Sigma$ there holds $\{\, \Phi(t; 0, x) : t \in (- \infty, \infty) \,\} \subset \widehat{\Sigma}$, part~(iv) follows.

    Part (v) is just \ref{P-order-convex}.
\end{proof}

Observe that for any $s \in \mytorus$ the system
\begin{equation} \label{eq:ODE-shifted}
    \dot{x}_i = x_i g_i(t + s, x), \quad i \in \{1, \dots, d \}, \ x \in C_{+},
\end{equation}
satisfies all the assumptions \ref{Cara}, \ref{0-repel}, and~\ref{new-H3} or~\ref{tilde-H3}, and \ref{competitive} or \ref{strongly-competitive}.  Solutions of~\eqref{eq:ODE-shifted} are just solutions of~\eqref{eq:ODE} with $t$ replaced with $t + s$.

Denote by $\cP_s$ the analogue of $\cP$ for~\eqref{eq:ODE-shifted}.  It is immediate that $\cP_s = \Phi(T + s; s, \cdot)$.

By the characterization given in Theorem~\ref{thm:attractor-ODE-ii}, the compact attractor of bounded sets in $C_{+}$ for the dynamical system $(\cP_s^n)_{n = 1}^{\infty}$ equals just the section $\widetilde{\Gamma}_s$.  Further, by the characterization given in Theorem~\ref{thm:tilde-Sigma-attractor-characterization}, the carrying simplex for $(\cP_s^n)_{n = 1}^{\infty}$ equals the section $\widetilde{\Sigma}_s$.

\begin{theorem}~\
\label{thm:tilde-Sigma-further}
Let $s \in \mytorus$.  The following holds.
      \begin{enumerate}[label=\textup{(\roman*)}, ref=\textup{\ref{thm:tilde-Sigma-further}}\textup{(\roman*)}]
        \item
        \label{thm:tilde-Sigma-further-unordered}
        $\widetilde{\Sigma}_s$ is weakly unordered \textup{[}under~\ref{strongly-competitive} $\widetilde{\Sigma}_s$ is unordered\textup{]}.
        \item
        \label{thm:tilde-Sigma-further-radial}
        $\widetilde{\Sigma}_s$ is homeomorphic via radial projection to the $(d-1)$\nobreakdash-\hspace{0pt}dimensional standard probability simplex $\Delta$.
        \item
        \label{thm:tilde-Sigma-further-char_of_differ}
        $\widetilde{\Gamma}_s \setminus \widetilde{\Sigma}_s$ is characterized as the set of those  $x \in C_{+}$ for which $\lim\limits_{t \to - \infty} \Phi(t;s,x) = 0$;
        \item
        \label{thm:tilde-Sigma-further-characterization}
        $\widetilde{\Sigma}_s$ is characterized as the set of those $x \in C_{+}$ for which $\{\, \Phi(t; s, x) : t \in (- \infty,\infty) \,\}$ is bounded and bounded away from $0$;
        \item
        \label{thm:tilde-Sigma-further-attractor-order-convex}
        $\widetilde{\Sigma}_s$ is the boundary \textup{(}relative to $C_{+}$\textup{)} of $\widetilde{\Gamma}_s$ and $\widetilde{\Gamma}_s$ is order convex; in~particular, $\widetilde{\Gamma}_s = \{\, (s, \alpha x) :  \alpha \in [0, 1], \ x \in \widetilde{\Sigma}_s \,\}$.
    \end{enumerate}
\end{theorem}
\begin{proof}[Indication of proof]
    In view of the above remarks, parts (i) and~(ii) follow directly from \ref{P-unordered} and \ref{P-radial-proj}, respectively, whereas the remaining parts are consequences of Theorem~\ref{thm:tilde-Sigma-attractor}.
\end{proof}
In view of Theorem~\ref{thm:tilde-Sigma-further-radial} we define, for each $s \in \mytorus$, the (unique) \emph{radial representation} of $\widetilde{\Sigma}_s$ as the continuous function $R^{*}_s \colon \Delta \to (0, \infty)$ such that
\begin{equation*}
    \widetilde{\Sigma}_s = \{\, R^{*}_s(u) u : u \in \Delta \,\}.
\end{equation*}
Instead of $R^{*}_0$ we frequently write $R^{*}$.
\begin{lemma}
\label{lm:radial}
    The radial representation $R^{*}_s$ depends continuously on $s \in \mytorus$, uniformly in $u \in \Delta$.
\end{lemma}
\begin{proof}
    By Proposition~\ref{prop:sections-continuity-i}, the map $\mytorus \ni s \mapsto \widetilde{\Sigma}_s$ is continuous.  Let $s_m \to s$ in $\mytorus$ and $u_m \to u$ in $\Delta$, as $m \to \infty$.  It follows from part~(PK2) of the definition of the Kuratowski convergence and the uniqueness of the radial representation that $R^{*}_{s_m}(u_m) u_m \to R^{*}_s(u) u$ as $m \to \infty$.  And, since $\Delta$ is compact, continuous convergence is equivalent to uniform convergence.  For similar reasoning, however in a slightly different context, see~\cite[the proof of Prop.~3.3(a)]{B-M-JDEA}.
\end{proof}

\begin{proposition}
\label{prop:boundary}
    $\widetilde{\Sigma}$ equals the boundary of $\widetilde{\Gamma}$ in $\mytorus \times C_{+}$.
\end{proposition}
\begin{proof}
    Let $(s, x) \in \widetilde{\Sigma}$.  Then, by Theorem~\ref{thm:tilde-Sigma-further-attractor-order-convex}, $(s, (1 + \epsilon)x) \notin \widetilde{\Gamma}$ for any $\epsilon > 0$, which shows that $(s, x)$ belongs to the boundary of $\widetilde{\Gamma}$ in $\mytorus \times C_{+}$.

    $\widetilde{\Gamma}$ is compact, so it is closed in $\mytorus \times C_{+}$, consequently, its boundary in $\mytorus \times C_{+}$ is a subset of $\widetilde{\Gamma}$.  Therefore, we need to exclude the case that some $(s, x) \in \widetilde{\Gamma} \setminus \widetilde{\Sigma}$ belongs to the boundary.  To do that, take $(s, x)$ to be so.  In view of Theorem~\ref{thm:tilde-Sigma-further-attractor-order-convex} and the definition of radial representation, there are $u \in \Delta$ and $\alpha \in (0, 1)$ such that $x = \alpha R_s(u) u$.  It follows from Lemma~\ref{lm:radial} that there is $\delta > 0$ such that if $\tilde{d}((\bar{s},y), (s,x)) < \delta$ then $y = \beta R_{\bar{s}}(v) v$ for some $v \in \Delta$ sufficiently close to $u$  and some $\beta \in (0, (1 + \alpha)/2)$, which means that the open ball centred at $(s, x)$ with radius $\delta$ is contained in $\widetilde{\Sigma}$.  Hence $(s, x)$ cannot belong to the boundary of $\widetilde{\Gamma}$ in $\mytorus \times C_{+}$.
\end{proof}

\begin{definition}[Periodic weak carrying simplex / periodic carrying simplex]
We shall refer to $\widetilde{\Sigma}$ obtained above as the \emph{time\nobreakdash-\hspace{0pt}periodic weak carrying simplex} \textup{(}or, under~\ref{strongly-competitive}, the \emph{time\nobreakdash-\hspace{0pt}periodic carrying simplex}\textup{)}.
\end{definition}

\smallskip
Indeed, $\widetilde{\Sigma}$ can be considered a generalization of the notion of carrying simplex $\Sigma$ to the case when the map $F \colon C_{+} \to C_{+}$ is replaced by $\mathbb{S}(T, \cdot)  \colon \mytorus \times C_{+} \to \mytorus \times  C_{+}$.  The latter is an example of a skew-product map:  the first coordinate (corresponding to the \emph{base} $\mytorus$) of the action depends only on the first coordinate, whereas the second coordinate (corresponding to the \emph{fibre} $C_{+}$) depends on both coordinates.  Further, for each $s \in \mytorus$ the map
\begin{equation*}
    \left[\, C_{+} \ni x \mapsto \Phi(s^{\dagger} + T; s^{\dagger}, x) \in C_{+} \,\right]
\end{equation*}
satisfies (A1)--(A3) and (\textoverline{A4}) [resp.\ (A1)--(A2), (A3)\textprime\ and (\textoverline{A4}\textprime)]

\smallskip
In a couple of results below we will consider some fixed $s \in \mytorus$, identified with its unique representative in $[0, T)$.
\begin{theorem}
\label{thm:construction-1}
    \begin{enumerate}
        \item[\textup{(i)}]
        There exists $\varepsilon > 0$ such that
        \begin{equation*}
            \widetilde{\Sigma}_s = \lim\limits_{n \to \infty} \Phi(n T + s; 0, \varepsilon \Delta)
        \end{equation*}
        for any $s \in [0, T)$.
        \item[\textup{(ii)}]
        \begin{equation*}
            \widetilde{\Sigma}_s = \lim\limits_{n \to \infty} \Phi(n T + s; 0, H)
        \end{equation*}
        for any $s \in [0, T)$.
    \end{enumerate}
\end{theorem}
\begin{proof}
    By Theorem~\ref{thm:construction-0}, there exists $\varepsilon > 0$ such that $\widetilde{\Sigma}_0 = \lim\limits_{n \to \infty} \Phi(n T; 0, \varepsilon \Delta)$.

    Fix $s \in [0, T)$.  We have, by~\eqref{eq:process} and the definition of $\cP^n$,
    \begin{equation*}
        \Phi(n T + s; 0, \varepsilon \Delta) = \Phi(n T + s; n T, \Phi(n T; 0, \varepsilon \Delta)) = \Phi(n T + s; n T, \cP^n(\varepsilon \Delta)),
    \end{equation*}
    which is equal, by~\eqref{eq:process-periodic}, to
    \begin{equation*}
        \Phi(s; 0, \cP^n(\varepsilon \Delta)).
    \end{equation*}
    Application of Proposition~\ref{prop:hyperspace-continuous} gives that $(\Phi(s,0; \cdot){\restriction}_{\Lambda})^{\star}$ is continuous on $\mathcal{P}(\Lambda)$, from which it follows that
    \begin{equation*}
        \widetilde{\Sigma}_s = \Phi(s; 0, \lim_{n \to \infty} \cP^n(\varepsilon \Delta)) = \lim\limits_{n \to \infty} \Phi(n T + s; 0, \varepsilon \Delta).
    \end{equation*}

    A proof of part~(ii) goes in much the same way.
\end{proof}
In the rest of the present subsection $\varepsilon > 0$ is as in Theorem~\ref{thm:construction-0}

\smallskip
There is also another way in which $\Phi(n T + s; 0, \varepsilon \Delta)$ [$\Phi(n T + s; 0, H)$] converge to $\widetilde{\Sigma}_s$.

\begin{lemma}\label{lm:aux-add-1}
   \begin{enumerate}
       \item[\textup{(i)}]
       For each $n \in \NN$ and each $s \in \mytorus$ we can write
           \begin{equation*}
               \Phi(n T + s; 0, \varepsilon \Delta) = \{\, R_{n T + s}(u) u : u \in \Delta \,\},
           \end{equation*}
        where $R_{n T + s} \colon \Delta \to (0, \infty)$ is continuous.
        \item[\textup{(ii)}]
        For each $n \in \NN$ and each $s \in \mytorus$ we can write
           \begin{equation*}
               \Phi(n T + s; 0, H) = \{\, R^{n T + s}(u) u : u \in \Delta \,\},
           \end{equation*}
        where $R^{n T + s} \colon \Delta \to (0, \infty)$ is continuous.
   \end{enumerate}
\end{lemma}
We will call $R_{n T + s}$ [$R^{n T + s}$] the \emph{radial representations} of the corresponding sets.
\begin{proof}[Indication of proof]
    We will prove (i), the proof of (ii) being similar.

    By~\cite[Lemma~3.10]{B-M-JDEA}, $\Phi(n T; 0, \varepsilon \Delta)$ is weakly unordered for each $n \in \NN$.  To show that $\Phi(n T + s; 0, \varepsilon \Delta)$ is weakly unordered we apply the approach as in the first two paragraphs of the proof of~\cite[Lemma~3.10]{B-M-JDEA}.  It suffices to show that $\Phi(n T + s; n T, \cdot) = \Phi(s; 0, \cdot)$ is weakly retrotone, which follows from Proposition~\ref{prop:retrotone}.
\end{proof}

Let $R \colon \Delta \to (0, \infty)$ be any radial representation.  Observe that there holds
\begin{equation} \label{eq:radial}
    u = \frac{R(u)u}{\lVert R(u) u \rVert_1}, \qquad u \in \Delta.
\end{equation}

We will consider now what happens when we allow $s \in \mytorus$ to change.  As the choice of the representative $s^{\dagger} \in [0, T)$ of $s \in \mytorus$ is discontinuous at $0 \mod T$, some care is needed.
\begin{theorem} \label{thm:construction-2}
    Let $(s_n)_{n = 0}^{\infty} \subset \mytorus$ converge to $s \in \mytorus$.  Then
    \begin{enumerate}
        \item[\textup{(i)}]
        \begin{equation*}
            \lim_{n \to \infty} \Phi(n T + s^{\dagger}_n; 0,\varepsilon \Delta)  = \widetilde{\Sigma}_s,
        \end{equation*}
        \item[\textup{(ii)}]
        \begin{equation*}
            \lim_{n \to \infty} \Phi(n T + s^{\dagger}_n; 0, H)  = \widetilde{\Sigma}_s
        \end{equation*}
    \end{enumerate}
    where $s^{\dagger}_n$ is the unique representative of $s_n$ in $[0, T)$.
\end{theorem}
\begin{proof}
    We prove only (i), the proof of (ii) being similar.  By Proposition~\ref{prop:Kuratowski-Hausdorff}, it suffices to prove the following:
    \begin{itemize}
        \item[(I)]
        for each $w \in \widetilde{\Sigma}_s$ there is a sequence $(u^{(n)})_{n = 0}^{\infty} \subset \varepsilon \Delta$ such that
        \begin{equation*}
            \lim_{n \to \infty} \Phi(n T + s^{\dagger}_n; 0,u^{(n)})  = w;
        \end{equation*}
        \item[(II)]
        for any sequences $(n_k)_{k = 0}^{\infty}$ and $(u^{(n_k)})_{k = 0}^{\infty} \subset \varepsilon \Delta$ such that $\lim\limits_{k \to \infty} n_k = \infty$ and  $\lim\limits_{k \to \infty} \Phi(n_k T + s^{\dagger}_{n_k}; 0,u^{(n_k)}) = w$ there holds $w \in \widetilde{\Sigma}_s$.
    \end{itemize}
    First of all, observe that for any $u \in C$ and any $n \in \NN$ there holds
    \begin{equation} \label{eq:aux-1}
        \Phi(n T + s_n^{\dagger}; 0, u) \stackrel{\eqref{eq:process}}{=} \Phi(n T + s_n^{\dagger}; n T, \Phi(n T; 0, u)) \stackrel{\eqref{eq:process-periodic}}{=} \Phi(s_n^{\dagger}; 0, \Phi(n T; 0 , u)).
    \end{equation}
    Consider first the case that $\lim\limits_{n \to \infty} s_n^{\dagger} =: s^{\dagger}$ exists (as a number from $[0, T]$).  Fix $w \in \widetilde{\Sigma}_s$.  As $\widetilde{\Sigma}_s = \Phi(s^{\dagger}; 0, \Sigma)$ and, by Lemma~\ref{lm:homeo}, $\Phi(s^{\dagger}; 0,\cdot)\!\!\restriction_{\Sigma}$ is a homeomorphism of $\Sigma$ onto its image, there is a unique $v \in \Sigma$ such that $\Phi(s^{\dagger}; 0, v) = w$.  Theorem~\ref{thm:construction-0}(i) gives, via (PK1), the existence of $(u^{(n)})_{n = 1}^{\infty} \subset \varepsilon \Delta$ such  that $\Phi(n T; 0, u^{(n)})$ converge, as $n \to \infty$, to $v$.  It follows from Theorem~\ref{thm:cont-depend} via~\eqref{eq:aux-1} that
    \begin{equation*}
        \lim_{n \to \infty} \Phi(n T + s^{\dagger}_n; 0, u^{(n)}) = \lim_{n \to \infty} \Phi(s_n^{\dagger}; 0, \Phi(n T; 0 , u^{(n)})) = \Phi(s^{\dagger}; 0, v) = w.
    \end{equation*}

    Assume now that $\lim\limits_{n \to \infty} s_n^{\dagger} =: s^{\dagger}$ does not exist.  This means that the sequence $(s_n)_{n = 0}^{\infty}$ has exactly two accumulation points, namely $0$ and $T$.  Fix $w \in \Sigma$, and let $v \in \Sigma$ be the unique element of $\Sigma$ such that $\Phi(T; 0, v) = w$.  Theorem~\ref{thm:construction-0}(i) gives, via (PK1), the existence of
    \begin{itemize}
        \item
        $(\hat{u}^{(n)})_{n = 1}^{\infty} \subset \varepsilon \Delta$ such that $\Phi(n T; 0, \hat{u}^{(n)})$ converge, as $n \to \infty$, to $w$,
        \item
        $(\check{u}^{(n)})_{n = 1}^{\infty} \subset \varepsilon \Delta$ such that $\Phi(n T; 0, \check{u}^{(n)})$ converge, as $n \to \infty$, to $v$.
    \end{itemize}
    We define
    \begin{equation*}
        u^{(n)} := \begin{cases}
            \hat{u}^{(n)} & \text{if } s_n^{\dagger} \in [0, T/2]
            \\
            \check{u}^{(n)} & \text{if } s_n^{\dagger} \in (T/2, T).
        \end{cases}
    \end{equation*}
    Let $(s_{n_k})_{k = 0}^{\infty}$ be a subsequence for which $\lim\limits_{k \to \infty} s_{n_k}^{\dagger} = 0$.  Then for sufficiently large $k$ we have $u^{(n_k)} = \hat{u}^{(n_k)}$.  It follows from Theorem~\ref{thm:cont-depend} via~\eqref{eq:aux-1} that
    \begin{equation*}
        \lim_{k\to \infty} \Phi(n_k T + s^{\dagger}_{n_k}; 0,u^{(n_k)}) = \lim_{k \to \infty} \Phi(s_{n_k}^{\dagger}; 0, \Phi(n_k T; 0, \hat{u}^{(n_k)})) = \Phi(0; 0, w) = w.
    \end{equation*}
    Let $(s_{n_l})_{l = 0}^{\infty}$ be a subsequence for which $\lim\limits_{l \to \infty} s_{n_l}^{\dagger} = T$.  Then for sufficiently large $l$ we have $u^{(n_l)} = \check{u}^{(n_l)}$.  It follows from Theorem~\ref{thm:cont-depend} via~\eqref{eq:aux-1} that
    \begin{equation*}
        \lim_{l\to \infty} \Phi(n_l T + s^{\dagger}_{n_l}; 0,u^{(n_l)}) = \lim_{lk \to \infty} \Phi(s_{n_l}^{\dagger}; 0, \Phi(n_l T; 0, \check{u}^{(n_l)})) = \Phi(T; 0, v) = w.
    \end{equation*}
    As any convergent subsequence of $(s^{\dagger}_n)_{n = 0}^{\infty}$ tends either to $0$ or to $T$, (I) is proved.

    We proceed to showing (II).  Let $(n_k)_{k = 0}^{\infty}$ and $(u^{(n_k)})_{k = 0}^{\infty} \subset \varepsilon \Delta$ be such that $\lim\limits_{k \to \infty} n_k = \infty$ and  $\lim\limits_{k \to \infty} \Phi(n_k T + s^{\dagger}_{n_k}; 0,u^{(n_k)})  = w$.  We wish to prove that $w \in \widetilde{\Sigma}_s$.  By passing to a subsequence we can assume that $(s^{\dagger}_{n_k})_{k = 0}^{\infty}$ converges to $s^{\dagger}$ as $k \to \infty$.  We have
    \begin{multline*}
        \Phi((n_k + 1) T; 0, u^{(n_k)}) \stackrel{\eqref{eq:process}}{=} \Phi((n_k + 1) T; n_k T + s^{\dagger}_{n_k}; \Phi(n_k T + s^{\dagger}_{n_k}; 0,u^{(n_k)}))
        \\
        \stackrel{\eqref{eq:process-periodic}}{=} \Phi(T; s^{\dagger}_{n_k}; \Phi(n_k T + s^{\dagger}_{n_k}; 0,u^{(n_k)})).
    \end{multline*}
    It follows then from Theorem~\ref{thm:cont-depend} that $\lim\limits_{k \to \infty} \Phi((n_k + 1) T; 0, u^{(n_k)})$ exists (denote it by $z$) and equals $\Phi(T; s^{\dagger}, w)$.  But, as $\Sigma$ is the Kuratowski limit of the sequence $(\Phi(n T; 0, \varepsilon \Delta))_{n = 0}^{\infty}$, it follows from (PK2) that $z \in \Sigma$.  Applying~\eqref{eq:process-inverse} we obtain $w = \Phi(s^{\dagger}; T, z)$.  Let $\zeta$ be the unique element of $\Sigma$ such that $\Phi(T; 0, \zeta) = z$.  Again by~\eqref{eq:process}, $w = \Phi(s^{\dagger}; 0, \zeta)$, that is, $w \in \widetilde{\Sigma}_s$.
\end{proof}

\begin{proposition} \label{prop:construction-2}
    Let $(s_n)_{n = 0}^{\infty} \subset \mytorus$ converge to $s \in \mytorus$.  Then
    \begin{enumerate}
        \item[\textup{(i)}]
        the functions $R_{n T + s_n} \colon \Delta \to (0, \infty)$ converge uniformly, as $n \to \infty$, to $R^{*}_s$;
        \item[\textup{(ii)}]
        the functions $R^{n T + s_n} \colon \Delta \to (0, \infty)$ converge uniformly, as $n \to \infty$, to $R^{*}_s$.
    \end{enumerate}
\end{proposition}
\begin{proof}
    As $\Delta$ is compact and all the functions considered are continuous, it suffices to show that if $(u^{(n)})_{n = 0}^{\infty} \subset \Delta$ is such that $\lim\limits_{n \to \infty} u^{(n)} = u \in \Delta$ then $\lim\limits_{n \to \infty} R_{n T + s_n}(u^{(n)}) = R(u)$.  By passing to a subsequence we can assume that $\lim\limits_{n \to \infty} s^{\dagger}_n$ exists as an element of $[0, T]$.  Denote this limit by $s^{\dagger}$.  As, for any $n \in \NN$, $R_{n T + s_n}(u^{(n)}) u^{(n)}$ belongs to the compact set $\Phi([0, T]; 0, \Lambda)$, by passing (perhaps a second time) to a subsequence we can assume that $w := R_{n_k T + s_{n_k}}(u^{(n_k)}) u^{(n_k)}$ exists.  By Theorem~\ref{thm:construction-2} (see (PK2)), $w \in \widetilde{\Sigma}_s$.  It follows from~\eqref{eq:radial} applied to $R^{*}_s$ that $P(w) = u$, so $R_{n_k T + s_{n_k}}(u^{(n_k)}) u^{(n_k)}$ converge, as $k \to \infty$, to $R^{*}_{s}(u) u$.   Consequently, $R_{n_k T + s_{n_k}}(u^{(n_k)}) = \lVert R_{n_k T + s_{n_k}}(u^{(n_k)}) u^{(n_k)} \rVert_1$ converge, as $k \to \infty$, to $\lVert R^{*}_{s}(u) u \rVert_1 = R^{*}_{s}(u)$.  Since all the accumulation points of the sequence $(R(u^{(n)}))_{n = 0}^{\infty}$ equal $R(u)$, we conclude that $\lim\limits_{n \to  \infty} R_n(u^{(n)}) = R(u)$.   This finishes the proof in the case when $s^{\dagger}$ exists.  In the other case, there are two accumulation points of $(u^{(n)})_{n = 0}^{\infty}$, namely $0$ and $T$.  The previous reasoning holds in the same form for any subsequence $(s_{n_k})_{k = 0}^{\infty}$ for which the representatives $s_{n_k}^{\dagger}$ converge, whether to $0$ or to $T$.

    This proves (i).  The proof of (ii) is analogous.
\end{proof}

\section{Conjugacy between the process on $\widetilde{\Sigma}$ and its orthogonal projection}
\label{sect:conjugacy}
\begin{lemma}
\label{lm:proj-sections-continuity}
    The assignment
    \begin{equation*}
        [\, \mytorus \ni s \mapsto \Pi(\widetilde{\Sigma}_s) \in 2^{V} \,]
    \end{equation*}
    is continuous.
\end{lemma}
 (Here $V = \mathbf{e}^\perp = \{x\in \real^d : \sum\limits_{i=1}^d x_i=0\}$).
\begin{proof}
    The proof is an exact copy of the proof of Proposition~\ref{prop:sections-continuity-i}.
\end{proof}

The purpose of the present subsection is to show that $\mathbb{S}$ restricted to $\widetilde{\Sigma}$ is conjugate to (a restriction of) some process generated on $V$ by a time\nobreakdash-\hspace{0pt}periodic system of Carathéodory ODEs.

We (at~first, formally) define, for $s, t \in \RR$,
\begin{equation*}
    \Psi(t; s, v) := \bigl(\Pi \circ \Phi(t;s, \cdot) \circ (\Pi{\restriction}_{\widetilde{\Sigma}_s})^{-1}\bigr)(t; s, v).
\end{equation*}

In order $\Psi$ to be well defined, we need to have $x \in \Pi(\widetilde{\Sigma}_s)$.  Then $\Psi(t; s, v) \in \Pi(\widetilde{\Sigma}_t)$.  As a consequence of~\eqref{eq:process} and the $\mathbb{S}$\nobreakdash-\hspace{0pt}invariance of $\widetilde{\Sigma}$ we have
\begin{equation*}
    \Psi(t_2; t_1, \Phi(t_1; s, v)) = \Psi(t_2; s, v), \quad s, t_1, t_2 \in \RR, \ v \in \Pi(\widetilde{\Sigma}_s).
\end{equation*}
Further, by~\eqref{eq:process-periodic} we have
\begin{equation*}
    \Psi(t; s, v) = \Psi(t + T; s + T, v), \quad s, t \in \RR, \ v \in \Pi(\widetilde{\Sigma}_s).
\end{equation*}
We will embed $\Psi$ into a periodic process $\widetilde{\Psi}$ on $V$ generated by time\nobreakdash-\hspace{0pt}periodic system of Carathéodory ODEs,
\begin{equation*}
    \dot{v} = \widecheck{H}(t, v), \quad v \in V.
\end{equation*}

For the remainder of the subsection we put $K := \bigcup\limits_{s \in \mytorus} \widetilde{\Sigma}_s$.  By Proposition~\ref{prop:sections-continuity-ii}, $K$ is compact.

Let $s \in \mytorus$ be not an exceptional value.  We define $\widecheck{G}(s, \cdot): \Pi(\widetilde{\Sigma}_s) \to V$ by the formula
\begin{equation*}
    \widecheck{G}(s, v) := \Pi \circ G(s, \cdot) \circ (\Pi{\restriction}_{\widetilde{\Sigma}_s})^{-1}(s, {v}), \quad {v} \in \Pi(\widetilde{\Sigma}_s).
\end{equation*}
By~\cite[Lem.~3.21]{B-M-JDEA}, $(\Pi{\restriction}_{\widetilde{\Sigma}_s})^{-1}$ satisfies the Lipschitz property with coefficient $\sqrt{1 + d}$.  We obtain from \ref{Cara-Lipschitz} that $\widecheck{G}_i(s, \cdot)$ satisfies, for each $1 \le i \le d$, the Lipschitz property with coefficient $\sqrt{1 + d} \, {\bigl(\mu_K(s) + \sup\{\, \norm{\xi} : \xi \in K \,\} \, \nu_K(s)\bigr)}$, see~\eqref{eq:Lipschitz}.

Let $(s_0, {v}_0) \in \widetilde{\Sigma}$.  For Lebesgue-a.e.\ $s \in \mytorus$ there holds
\begin{equation*}
    {\DD_1} \Phi(s; s_0, {v}_0) = G(s, \Phi(s; s_0, {v}_0)),
\end{equation*}
where $\DD_1$ denotes the derivative with respect to the first variable.  Applying $\Pi$ to both sides of the above equality and noticing that $\Pi$ and $\DD_1$ commute we obtain
\begin{equation*}
    {\DD_1} \bigl((\Pi \circ \Phi)(s; s_0, {v}_0)\bigr) = (\Pi \circ G(s, \cdot))(\Phi(s; s_0, {v}_0)),
\end{equation*}
and, further,
\begin{equation*}
    {\DD_1} \bigl((\Pi \circ \Phi)(s; s_0, {v}_0)\bigr) = (\Pi \circ G(s, \cdot) \circ (\Pi{\restriction}_{\widetilde{\Sigma}_s})^{-1})\bigl((\Pi \circ \Phi)(s; s_0, {v}_0)\bigr),
\end{equation*}
that is,
\begin{equation*}
    {\DD_1} \bigl((\Pi \circ \Phi)(s; s_0, {v}_0)\bigr) = \widecheck{G}(s, (\Pi \circ \Phi)(s; s_0, {v}_0)).
\end{equation*}
By slight abuse of language, we can say that $\Psi$ is the process generated by the time\nobreakdash-\hspace{0pt}periodic system of Carathéodory ODEs $\dot{v} = \widecheck{G}(t, v)$.

Our construction of the extension $\widecheck{H}$ of $\widecheck{G}$ to the whole of $\mytorus \times V$ follows that in the proof of~\cite[Thm.~1]{McSh}.  Namely, for any $s \in \mytorus$ we extend $\widecheck{G}(s, \cdot) \colon \Pi(\widetilde{\Sigma}_s) \to V$ to a  globally Lipschitz $\widecheck{H}(s, \cdot) \colon V \to V$, and we do it in such a way that for each $v \in V$ the map $\widecheck{H}(s, \cdot) \colon \mytorus \to V$ is measurable.

Let $\{\xi^{({m})}\}_{{m} = 1}^{\infty}$ be a dense subset of $\Pi(\widetilde{\Sigma}_0)$.  For each ${m} \in \NN$ put
\begin{equation*}
  \xi_s^{({m})} := (\Pi \circ \Phi(s; 0, \cdot) \circ (\Pi{\restriction}_{\widetilde{\Sigma}_0})^{-1})(\xi^{({m})}), \quad s \in \mytorus.
\end{equation*}
For each $s \in \mytorus$ the mapping $\Pi \circ \Phi(s; 0, \cdot) \circ (\Pi{\restriction}_{\widetilde{\Sigma}_0})^{-1}$ is a homeomorphism of $\Pi(\widetilde{\Sigma}_0)$ onto $\Pi(\widetilde{\Sigma}_s)$.  Consequently, $\{\xi_s^{({m})}\}_{{m} = 1}^{\infty}$ is a dense subset of $\Pi(\widetilde{\Sigma}_s)$.  Moreover, the assignment $\bigl[\, \mytorus \ni s \mapsto \xi_s^{({m})} \in V \,\bigr]$ is continuous, for each ${m} \in \NN$.

For a non-exceptional $s \in \mytorus$, denote by $\lambda(s)$ a Lipschitz coefficient of $\widetilde{G}(s, \cdot)$.  As mentioned before, for $\lambda(s)$ we can take $\sqrt{1 + d} \, \bigl(\mu_K(s) + \sup\{\, \norm{\xi} : \xi \in K \,\} \, \nu_K(s)\bigr)$.  As a consequence of~\ref{Cara-bounded} and~\ref{Cara-Lipschitz}, $\lambda(\cdot)$ belongs to $L_1([0, T])$.

For $1 \le i \le d$, a non-exceptional $s \in \mytorus$ and $m \in \NN$ we define a function $\overline{H}_i \colon \mytorus \times V \to \RR$ by
\begin{equation*}
        \overline{H}^{(m)}_i(s, v) := \widecheck{G}_i(s, \xi_s^{(m)}) - \lambda(s) \, \norm{v - \xi_s^{(m)}}.
\end{equation*}
Further, we set
\begin{equation*}
    \overline{H}_i(s, v) := \sup\{\, \overline{H}^{(m)}_i(s, v): m \in \NN \,\}, \quad v \in V.
\end{equation*}
We claim that $\overline{H}_i(s, \cdot)$ is Lipschitz, with Lipschitz constant $\le \lambda(s)$.  Indeed, let $v, w \in V$ and assume for definiteness that $\overline{H}_i(s, w) \ge \overline{H}_i(s, v)$.  We estimate
\begin{equation*}
    \begin{aligned}
        0 & \le \overline{H}_i(s, w) - \overline{H}_i(s, v)
        \\
        & = \sup\{\, \widecheck{G}_i(s, \xi_s^{(m)}) - \lambda(s) \, \norm{w - \xi_s^{(m)}}: m \in \NN \,\}
        \\
        & \phantom{\le} - \sup\{\, \widecheck{G}_i(s, \xi_s^{(m)}) - \lambda(s) \, \norm{v - \xi_s^{(m)}}: m \in \NN \,\}
        \\
        & \le \sup\{\, \bigl(\widecheck{G}_i(s, \xi_s^{(m)}) - \lambda(s) \, \norm{w - \xi_s^{(m)}} \bigr) - \bigl( \widecheck{G}_i(s, \xi_s^{(m)}) - \lambda(s) \, \norm{v - \xi_s^{(m)}} \bigr): m \in \NN \,\}
        \\
        & = \sup\{\, \lambda(s) \bigl( \norm{v - \xi_s^{(m)}} - \norm{w - \xi_s^{(m)}} \bigr): m \in \NN \,\} \le \lambda(s) \norm{w - v}.
    \end{aligned}
\end{equation*}
For any $m \in \NN$ we have $\overline{H}^{(m)}_i(s, \xi_s^{(m)}) = \widecheck{G}_i(s, \xi_s^{(m)})$.  We claim that $\overline{H}_i(s, \xi_s^{(m)}) = \overline{H}^{(m)}_i(s, \xi_s^{(m)})$.  Indeed, if not then there is $p \in \NN$, $p \ne m$, such that $\overline{H}^{(p)}_i(s, \xi_s^{(m)}) > \overline{H}^{(m)}_i(s, \xi_s^{(m)})$.  Observe that
\begin{equation*}
    \widecheck{G}_i(s, \xi_s^{(p)}) = \overline{H}^{(p)}_i(s, \xi_s^{(m)}) + \lambda(s) \norm{\xi_s^{(m)} - \xi_s^{(p)}},
\end{equation*}
consequently $\widecheck{G}_i(s, \xi_s^{(p)}) - \widecheck{G}_i(s, \xi_s^{(m)}) = \overline{H}^{(p)}_i(s, \xi_s^{(m)}) - \overline{H}^{(m)}_i(s, \xi_s^{(m)}) + \lambda(s) \norm{\xi_s^{(m)} - \xi_s^{(p)}} > \lambda(s) \norm{\xi_s^{(m)} - \xi_s^{(p)}}$.  The case $\xi_s^{(m)} = \xi_s^{(p)}$ implies $0 > 0$, so $\norm{\xi_s^{(m)} - \xi_s^{(p)}} > 0$, which contradicts the fact that $\lambda(s)$ is a Lipschitz coefficient for $\widecheck{G}_i(s, \cdot)$.

Since the Lipschitz functions $\overline{H}_i(s, \cdot){\restriction}_{\Pi(\widetilde{\Sigma}_s)}$ and $\widecheck{G}_i(s, \cdot)$ take the same value at any member of the dense subset $\{\xi_s^{(m)}\}_{m = 1}^{\infty}$ of $\Pi(\widetilde{\Sigma}_s)$, there holds $\overline{H}_i(s, v) = \widecheck{G}_i(s, v)$ for any $v \in \Pi(\widetilde{\Sigma}_s)$.

For an exceptional $s \in \mytorus$ we just put $\overline{H}_i(s, v) = 0$ for all $v \in V$, $1 \le 1 \le d$.

For each ${v} \in V$, each $1 \le i \le d$ and each $m \in \NN$ the mapping ${\overline{H}}^{{(m)}}_i(\cdot, {v})$ is measurable, hence ${\overline{H}}_i(\cdot, {v})$, is measurable, too.

We define
\begin{equation*}
    \widecheck{H} := \Pi \circ (\overline{H}_1, \ldots, \overline{H}_d).
\end{equation*}
We list here properties of $\widecheck{H} \colon \mytorus \times V \to V$:
\begin{itemize}
    \item
    for each $v \in V$ the map $\widecheck{H}(\cdot, v) \colon \mytorus \to V$ is measurable,
    \item
    for a.e.\ $s \in \mytorus$ the map $\widecheck{H}(s, \cdot) \colon V \to V$ satisfies Lipschitz condition, with Lipschitz constant belonging to $L_1(\mytorus)$,
    \item
    for a.e.\ $s \in \mytorus$ there holds $\widecheck{H}(s, v) = \widecheck{G}(s, v)$ for all $v \in \Pi(\widetilde{\Sigma}_s)$.
\end{itemize}

Therefore, $\widecheck{H}$ is an extension of $\widecheck{G}$ to $\mytorus \times V$ generating a periodic local process $\widecheck{\Psi}$ on $V$.  By \cite[Thm.~II.3.2]{Reid}, the process $\widecheck{\Psi}$ is indeed global: $\widecheck{\Psi}(t; s, \cdot)$ is defined for any $s, t \in \RR$.  Consequently, $\widecheck{\Psi}$ can be legitimately called an extension of $\Psi$.

To conclude the present subsection we formulated what we have just obtained as the following.
\begin{theorem}
    In the notation introduced above there holds
    \begin{equation*}
        \Pi(\mySemi(t, (s, x))) =  \widecheck{\Psi}(t; s, \Pi x), \quad t \in \RR, \ s \in \mytorus, \ x \in \widetilde{\Sigma}_s,
    \end{equation*}
    and
    \begin{equation*}
        (\Pi{\restriction}_{\widetilde{\Sigma}_{s}})^{-1}(\widecheck{\Psi}(t ; s, v)) =  \mySemi(t, (s, (\Pi{\restriction}_{\widetilde{\Sigma}_{s}})^{-1}({v}))), \quad t \in \RR, \ s \in \mytorus, \ v \in \Pi(\widetilde{\Sigma}_s).
    \end{equation*}
    In other words, the diagram
    \begin{equation*}
        \begin{tikzcd}[column sep=huge] 
       \widetilde{\Sigma}_s \arrow{r}{\Phi(t;s, \cdot)} \arrow[swap]{d}{\Pi{\restriction}_{\widetilde{\Sigma}_s}} & \widetilde{\Sigma}_t \arrow{d}{\Pi{\restriction}_{\widetilde{\Sigma}_t}}&  \\
        \Pi(\widetilde{\Sigma}_s) \arrow{r}{\widecheck{\Psi}(t;s, \cdot){\restriction}_{\Pi(\widetilde{\Sigma}_s)}} & \Pi(\widetilde{\Sigma}_t)
        \end{tikzcd}
    \end{equation*}
    commutes.
\end{theorem}
We can say that (a sort of) topological conjugacy holds between the flow on $\widetilde{\Sigma}$ and its orthogonal projection along $\ev$.

\subsubsection*{Acknowledgment}
The authors thank the anonymous referee for their careful reading.

\subsubsection*{Funding}
No funding.

\end{document}